\documentclass[final,leqno]{siamltex}

\usepackage{amsfonts,amssymb} 
\usepackage{amsmath} 
\usepackage{color} 
\usepackage{graphicx} 
\usepackage{epsfig,mathrsfs} 
\usepackage[bf,SL,BF]{subfigure} 
\usepackage{fancyhdr} 
\usepackage{CJK} 
\usepackage{caption} 
\usepackage{wrapfig} 
\usepackage{cases} 
\usepackage{bm}
\usepackage{algorithm}
\usepackage{algorithmic}
\newtheorem{example}{Example}[section] 

\title{Strong convergence order for the scheme of fractional diffusion equation driven by fractional Gaussion noise
\thanks{This work was supported by the National Natural Science Foundation of China under
	Grant No. 11671182, and the AI and Big Data Funds under Grant No. 2019620005000775.
}}

\author{Daxin Nie\footnotemark[2]\and Jing Sun\footnotemark[2]
\and Weihua Deng\footnotemark[2]\thanks{ School of Mathematics and Statistics, Gansu Key Laboratory of Applied Mathematics and Complex Systems, Lanzhou University, Lanzhou 730000, P.R. China (Email: dengwh@lzu.edu.cn).}
}
\begin{document}

\maketitle

\begin{abstract}
Fractional Gaussian noise models the time series with long-range dependence; when the Hurst index $H>1/2$, it has positive correlation reflecting a persistent autocorrelation structure. This paper studies the numerical method for solving stochastic fractional diffusion equation driven by fractional Gaussian noise. Using the operator theoretical approach, we present the regularity estimate of the mild solution and the fully discrete scheme with finite element approximation in space and backward Euler convolution quadrature in time. The $\mathcal{O}(\tau^{H-\rho\alpha})$ convergence rate in time and $\mathcal{O}(h^{\min(2,2-2\rho,\frac{H}{\alpha})})$ in space are obtained, showing the relationship between the regularity of noise and convergence rates, where $\rho$ is a parameter to measure the regularity of noise and $\alpha\in(0,1)$. Finally, numerical experiments are performed to support the theoretical results.

\end{abstract}


\pagestyle{myheadings}
\thispagestyle{plain}
\markboth{D. X. Nie, J. Sun and W. H. Deng}{CONVERGENCE FOR THE SCHEME OF FRACTIONAL DIFFUSION EQUATION}

\section{Introduction}

In some practical environments, external noises can not be ignored. Long-range dependence is observed in a wide range of fields, some of the time also accompanying self-similar properties, i.e., being invariant to changes of scale in time (or space).  Fractional Brownian motion (fBm) is the only self-similar continuous-time Gaussian process with stationary increments and long-range dependency structure \cite{Deng_2020}. It is built as a moving average of $dW(t)$ (white noise), in which past increments of $W(t)$ are weighted by the kernel $(t-t^\prime)^{H-1/2} $ and $H$ is a parameter satisfying $0<H<1$. This paper is concerned with fractional Gaussian noise (formal derivative of fBm) with persistent autocorrelation structure, i.e., $H>1/2$. The original unperturbed equation considered in this paper is for describing the anomalous diffusion with power-law waiting time.

More specifically, we study the numerical method for solving stochastic time-fractional diffusion equation driven by fractional Gaussian noise (external noise).
%
%
%
%
%
%
%
Let $D\subset \mathbb{R}^{d}$, $(d=1,2,3)$, be a convex polygonal domain and $G$ the solution of
\begin{equation}\label{equretosol0}
\left \{
\begin{split}
&\partial_{t} G+\!_0\partial^{1-\alpha}_tA G
=\dot{W}^{H}_{Q} \,~\qquad\quad {\rm in}\ D,\ t\in[0,T],\\
&G(\cdot,0)=G_0 \,\qquad\qquad\qquad\qquad {\rm in}\ D,\\
&G=0 \qquad\qquad\qquad\qquad\qquad\ \ \, {\rm on}\ \partial D,\ t\in[0,T],
\end{split}
\right .
\end{equation}
where $A=-\Delta$ with a zero Dirichlet boundary condition and its domain $\mathcal{D}(A)=H^{1}_{0}(D)\cap H^{2}(D)$; $\alpha\in(0,1)$; $\dot{W}^{H}_{Q}$ is fractional Gaussian noise and $W^{H}_{Q}$ is fractional Gaussian process with a covariance operator $Q$ (a self-adjoint, nonnegative, linear operator on $\mathbb{H}=L^{2}(D)$)  on a complete filtered probability space $(\Omega,\mathcal{F},\mathbb{P},\{\mathcal{F}_{t}\}_{t\geq 0})$ and Hurst index  $H\in (1/2,1)$; we assume that $A^{-\rho}Q^{1/2}$ with $\rho$ being a real number is a bounded operator in $\mathbb{H}$; and $\!_0\partial^{1-\alpha}_t$ is the Riemann-Liouville fractional derivative with $\alpha\in(0,1)$ defined by \cite{Podlubny_1999}
\begin{equation}
_{0}\partial^{1-\alpha}_tG=\frac{1}{\Gamma(\alpha)}\frac{\partial}{\partial t}\int^t_{0}(t-\xi)^{\alpha-1}G(\xi)d\xi.
\end{equation}


 In recent years, numerical method for the stochastic fractional diffusion equation driven by Gaussian noise has gained widespread concerns \cite{gunzburger_2018,gunzburger_2019,jin2019-1,kovacs_2014,kovacs_20142,li_2017,wu_2020}.
References \cite{kovacs_2014,kovacs_20142} present strong and weak convergence rates for a linear stochastic volterra type evolution equation driven
by an additive Gaussian noise. The works \cite{gunzburger_2018,gunzburger_2019} use the convolution quadrature and finite element method to discretize the stochastic time-fractional partial differential equations subject to additive space-time white noise. Reference \cite{jin2019-1} considers a fractional diffusion equation driven by the fractionally integrated Gaussian noise. The work \cite{li_2017} investigates the Galerkin finite element approximations for the stochastic space-time
fractional wave equations with an infinite dimensional additive noise.
As for the numerical approximations of the time fractional partial differential equations driven by fractional Gaussian noise, there are relatively few studies.
 Reference \cite{Li_2019} presents the Galerkin finite element semi-discrete scheme for semilinear stochastic time-tempered fractional wave equations driven by multiplicative Gaussian noise and fractional Gaussian noise and provides the error analyses when the noise is regular enough.

 In this paper, we use backward Euler convolution quadrature and finite element method to discretize the time fractional derivative and the space operator, respectively. Different from the Gaussian noise, fractional Gaussian noise is more regular in time, that is, the trajectory of fractional Brownian motion belongs to $C^{H}([0,T])$, so the corresponding estimate depends on the Hurst index $H$. But the approach adopted in \cite{Li_2019,Liu_2020} can't reflect the influence of Hurst index on regularity of the solutions and convergence rates.  Here, we take the appropriate weighted function and combine the operator theoretical approach to get the regularity estimate for the solution, i.e., $	(\mathbb{E}\|A^{\sigma}G\|_{\mathbb{H}}^{2})^{1/2}\leq C+Ct^{-(\sigma-q)\alpha}\|G_{0}\|_{\hat{H}^{q}(D)}$, $q\leq\sigma=\min(1,1-\rho,\frac{H}{\alpha}-\rho-\epsilon)$, and $	\left (\mathbb{E}\left \|\frac{G(t)-G(t-\tau)}{\tau^{\gamma}}\right \|_{\mathbb{H}}^{2}\right )^{1/2}\leq C+Ct^{-\gamma}\|G_{0}\|_{\mathbb{H}}$ with $\gamma<H-\rho\alpha$; see Theorem \ref{thmsobo}. Then, an $\mathcal{O}(\tau^{H-\alpha\rho})$ $(\rho\in[0,H/\alpha))$ convergence rate in time and an $\mathcal{O}(h^{\min(2,2-2\rho,\frac{2H}{\alpha}-2\rho)})$ $(\rho\in[-1,1/2])$ convergence rate in space can be got by skillful analyses; see Theorems \ref{thmGtime} and \ref{ThmGspace}. From the analyses, we show the relationship between the regularity of noise and convergence rates of the numerical scheme.

The paper is organized as follows. In the next section, we introduce some notations and facts on fractional Brownian motion, and then discuss the regularity of the mild solution of Eq. \eqref{equretosol0}. In Section 3, we use the backward Euler method to discretize the time fractional derivative and provide the error estimates for the semidiscrete scheme. Finite element method is used to discretize space operator and error estimates for the fully discrete scheme are proposed in Section 4. We confirm the  theoretically predicted convergence order by numerical examples in Section 5. The paper is concluded with some discussions in the last section. Throughout the paper, $C$ denotes a generic positive constant, whose value may differ at each occurrence.

\section{Regularity of the solution}
\subsection{Preliminaries}
 We first introduce some notations.
  Let  $ {(\varkappa_{j},\varphi_j)} $ be $A$'s eigenvalues ordered non-decreasingly and the corresponding eigenfunctions normalized in the $ L^2(D) $ norm, where $A=-\Delta$ with a zero Dirichlet boundary condition.  For any $ q\in \mathbb{R} $, denote the space $ \hat{H}^q(D)=\mathcal{D}(A^{q/2})$ \cite{Thomee_2006} with the norm given by
	\begin{equation*}
		|\mu|^2_{\hat{H}^q(D)}=\|A^{q/2}\mu\|_{L^2(D)}.
	\end{equation*}
	Thus $ \hat{H}^0(D)=L^2(D) $, $\hat{H}^1(D)=H^1_0(D)$, and $\hat{H}^2(D)=H^2(D)\bigcap H^1_0(D)$.
Below we  use the notation `$\tilde{~}$' for taking Laplace transform and let $\epsilon > 0$ be arbitrarily small.

Then we introduce some preliminary facts on fractional Brownian motion, which one can refer to \cite{Mishura_2008,Kloeden_1995,Prato_2014}. Denote $\mathcal{L}(\mathbb{U};\mathbb{V})$ as the Banach space of all bounded linear operators $\mathbb{U}\rightarrow \mathbb{V}$, where $\mathbb{U}$ and $\mathbb{V}$ are two separable Hilbert spaces with norms $\|\cdot\|_{\mathbb{U}}$, $\|\cdot\|_{\mathbb{V}}$, and inner products $(\cdot,\cdot)_{\mathbb{U}}$ and $(\cdot,\cdot)_{\mathbb{V}}$. We introduce  $\mathcal{L}_{2}(\mathbb{U};\mathbb{V}) \,(\subset \mathcal{L}(\mathbb{U};\mathbb{V}))$, which consists of all Hilbert-Schmidt operators with norm and inner product
\begin{equation*}
	\|T\|^{2}_{\mathcal{L}_{2}(\mathbb{U},\mathbb{V})}=\sum_{j\in \mathbb{N}}\|T\mu_{j}\|^{2}_{\mathbb{V}},\quad \langle S,T \rangle_{\mathcal{L}_{2}(\mathbb{U},\mathbb{V})}=\sum_{j\in\mathbb{N}}( S\mu_{j},T\mu_{j})_{\mathbb{V}},\quad S,T\in \mathcal{L}_{2}(\mathbb{U},\mathbb{V}),
\end{equation*}
where $\{\mu_{j}\}_{j\in \mathbb{N}}$ is the orthonormal basis in $\mathbb{U}$ and the above definitions are independent of the specific choice of orthonormal bases.
Let $\mathbb{H}=L^{2}(D)$ with inner product $(\cdot,\cdot)$ and covariance operator $Q$ be a  self-adjoint, nonnegative linear operator on $\mathbb{H}$. We abbreviate $\|\cdot\|_{\mathcal{L}(\mathbb{H},\mathbb{H})}$ as $\|\cdot\|$. Denote $\mathbb{H}_{0}=Q^{1/2}(\mathbb{H})$ with inner product $( \mu,\nu)_{\mathbb{H}_{0}}=( Q^{-1/2}\mu,Q^{-1/2}\nu)$ for $\mu,\nu\in\mathbb{H}_{0}$ and abbreviate $\langle \cdot,\cdot \rangle_{\mathcal{L}_{2}(\mathbb{H},\mathbb{H})}$, $\mathcal{L}_{2}(\mathbb{H},\mathbb{H})$ and $\mathcal{L}_{2}(\mathbb{H}_{0},\mathbb{H})$ as  $\langle \cdot,\cdot \rangle$, $\mathcal{L}_{2}$ and $\mathcal{L}_{2}^{0}$. The  fractional Brownian motion (fBm) with covariance operator $Q$ can be represented as
\begin{equation*}
W^{H}_{Q}(x,t)=\sum_{k=1}^{\infty}\sqrt{\varLambda_{k}}\phi_k(x)W^{H}_k(t),
\end{equation*}
where $(\phi_k)_{k\in\mathbb{N}}$ denotes an orthonormal basis of $\mathbb{H}$ consisting of the eigenfunctions of $Q$ with corresponding eigenvalues $(\varLambda_k)_{k\in\mathbb{N}}$ and $W^{H}_k$, $k=1,2,\ldots$ are independent one-dimensional fBm process with Hurst index $H$. We denote $W^{H}_{Q}(t)$ as $W^{H}_{Q}(x,t)$ and $\mathbb{E}$ as expectation in the following.

Furthermore, for $\kappa>0$ and $\pi/2<\theta<\pi$, we define sectors 
\begin{equation*}
\begin{aligned}
&\Sigma_{\theta}=\{z\in\mathbb{C}:z\neq 0,|\arg z|\leq \theta\},\ \Sigma_{\theta,\kappa}=\{z\in\mathbb{C}:|z|>\kappa,|\arg z|\leq \theta\},\\
\end{aligned}
\end{equation*}
and 
the contour $\Gamma_{\theta,\kappa}$ by
\begin{equation*}
\Gamma_{\theta,\kappa}=\{r e^{-\mathbf{i}\theta}: r\geq \kappa\}\bigcup\{\kappa e^{\mathbf{i}\psi}: |\psi|\leq \theta\}\bigcup\{r e^{\mathbf{i}\theta}: r\geq \kappa\},
\end{equation*}
where the circular arc is oriented counterclockwise and the two rays are oriented with an increasing imaginary part and $\mathbf{i}^2=-1$.

For one-dimensional fBm, the following holds.
\begin{lemma}[\cite{Mishura_2008,Kloeden_1995}]\label{Lemitoeq}
For $H>1/2$ and $f(t),g(t)\in L^2([0,T])$, we have
\begin{equation*}
\mathbb{E} \int_{0}^{t} f(s) d W^{H}(s)=0
\end{equation*}
and
\begin{equation*}
\mathbb{E} \left [\int_{0}^{t} f(s) d W^{H}(s) \int_{0}^{t} g(s) d W^{H}(s)\right ]=H(2 H-1) \int_{0}^{t} \int_{0}^{t} f(s) g(r)|s-r|^{2 H-2} d r d s,
\end{equation*}
where $W^H$ means one-dimensional fBm.
\end{lemma}

 Then, we provide a lemma which plays an important role in the proofs of this paper.

\begin{lemma}\label{eqcorlem}
Let	$f(t),g(t)\in \mathcal{L}^{0}_{2}$, and $H>1/2$. We have
\begin{equation*}
\begin{aligned}
&\mathbb{E} \left (\int_{0}^{t} f(s) d W^{H}_{Q}(s), \int_{0}^{t} g(s) d W^{H}_{Q}(s)\right )=\\
&\qquad \qquad\qquad \qquad\qquad \quad H(2H-1) \int_{0}^{t}\int_{0}^{t} \langle f(s)Q^{1/2},g(r)Q^{1/2}\rangle|r-s|^{2H-2}drds,
\end{aligned}
\end{equation*}
	where $(\cdot,\cdot)$ denotes $L^{2}$ inner product.
\end{lemma}
 \begin{proof}
Using Lemma \ref{Lemitoeq} and the definition of $W^{H}_{Q}$, one has
\begin{equation*}
	\begin{aligned}
		&\mathbb{E} \left (\int_{0}^{t} f(s) d W^{H}_{Q}(s), \int_{0}^{t} g(s) d W^{H}_{Q}(s)\right )\\
		= &\sum_{k=1}^{\infty}\mathbb{E} \left (\int_{0}^{t} f(s)\sqrt{\varLambda_{k}}\phi_k(x) d W^{H}_{k}(s), \int_{0}^{t} g(s)\sqrt{\varLambda_{k}} \phi_k(x)d W^{H}_{k}(s)\right )\\
		= &\sum_{k=1}^{\infty}\mathbb{E} \left [\int_{0}^{t} f(s)\sqrt{\varLambda_{k}}d W^{H}_{k}(s) \int_{0}^{t} g(s)\sqrt{\varLambda_{k}}d W^{H}_{k}(s)\right ]\\
		= &H(2H-1)\sum_{k=1}^{\infty} \int_{0}^{t}\int_{0}^{t} f(s)\sqrt{\varLambda_{k}}g(r)\sqrt{\varLambda_{k}}|r-s|^{2H-2}drds  \\
		= &H(2H-1)\sum_{k=1}^{\infty} \int_{0}^{t}\int_{0}^{t} \left (f(s)\sqrt{\varLambda_{k}}\phi_k(x),g(r)\sqrt{\varLambda_{k}}\phi_k(x)\right )|r-s|^{2H-2}drds  \\
		= &H(2H-1) \int_{0}^{t}\int_{0}^{t} \langle f(s)Q^{1/2},g(r)Q^{1/2}\rangle|r-s|^{2H-2}drds.
	\end{aligned}
\end{equation*}

\end{proof}

\subsection{A priori estimate of the solution}
It is easy to see that the solution of Eq. \eqref{equretosol0} can be decomposed into the solution of the deterministic problem
\begin{equation}\label{equretosol1}
\left \{
\begin{split}
&\partial_{t} v+\!_0\partial^{1-\alpha}_tAv=0
 \qquad\quad\quad~~\, {\rm in}\ D,\ t\in[0,T],\\
&v(\cdot,0)=G_0 \qquad\qquad\qquad\qquad {\rm in}\ D,\\
&v=0 \qquad\qquad\qquad\qquad\qquad\ \ \, {\rm on}\ \partial D,\ t\in[0,T],
\end{split}
\right .
\end{equation}
plus the solution of the stochastic problem
\begin{equation}\label{equretosol}
\left \{
\begin{split}
&\partial_{t} u+\!_0\partial^{1-\alpha}_tAu
=\dot{W}^{H}_{Q} \quad\quad~~ {\rm in}\ D,\ t\in[0,T],\\
&u(\cdot,0)=0 \qquad\qquad\qquad\qquad {\rm in}\ D,\\
&u=0 \qquad\quad\qquad\qquad\qquad\ \ \, {\rm on}\ \partial D,\ t\in[0,T].
\end{split}
\right .
\end{equation}
As for Eq. \eqref{equretosol1}, there have been many discussions; see, e.g., those  in \cite{cuesta_2006-1,jin2016-1,jin_2019_2,lin_2007,mclean_2015,sakamoto_2011}. So in this paper, we mainly focus on Eq. \eqref{equretosol}. 
To get the representation of the mild solution of Eq. \eqref{equretosol}, we introduce the operator 
\begin{equation*}
	E(t)=\frac{1}{2\pi\mathbf{i}}\int_{\Gamma_{\theta,\kappa}}e^{zt}z^{\alpha-1}(z^{\alpha}+A)^{-1}dz.
\end{equation*}
Thus by taking Laplace and inverse Laplace transforms, the mild solution of Eq. \eqref{equretosol} can be written as
\begin{equation}\label{eqrepsol}
u=\int_{0}^{t}E(t-s)dW^{H}_{Q}(s).
\end{equation}
According to the resolvent estimate $\|(z+A)^{-1}\|\leq C|z|$ for $z\in \Sigma_{\theta}$ \cite{lubich_1996}, the following estimates hold
\begin{equation*}
	\begin{aligned}
	&\|\tilde{E}(z)\|\leq C|z|^{-1}\quad \forall z\in \Sigma_{\theta},\\
	&\|A\tilde{E}(z)\|\leq C|z|^{\alpha-1}\quad \forall z\in \Sigma_{\theta},
	\end{aligned}
\end{equation*}
where $\tilde{E}(z)$ means the Laplace transform of $E$. Using interpolation property leads to
\begin{equation}\label{equresolvent}
	\|A^{\beta}\tilde{E}(z)\|\leq C|z|^{\beta\alpha-1}\quad \forall z\in \Sigma_{\theta}~{\rm and }~ \beta\in [0,1].
\end{equation}
According to the above estimates, the spatial regularity estimate of $u$ can be obtained.
\begin{theorem}\label{thmsobo}
	Let $u$ be the mild solution of Eq. $\eqref{equretosol}$ and $\|A^{-\rho}\|_{\mathcal{L}^{0}_{2}}<\infty$ with $\rho<\frac{H}{\alpha}$, where $\alpha\in(0,1)$. Then we have
	\begin{equation*}
		\mathbb{E}\|A^{\sigma}u\|^{2}_{\mathbb{H}}\leq C,
	\end{equation*}
	where $\sigma\leq\min(1-\rho,\frac{H}{\alpha}-\rho-\epsilon)$.
\end{theorem}
\begin{proof}
	Taking the expectation of the square of the $\mathbb{H}$-norm of $A^{\sigma}u$ and using \eqref{equretosol} give
\begin{equation*}
\begin{aligned}
\mathbb{E}\|A^{\sigma}u\|^{2}_{\mathbb{H}}=\mathbb{E}\left \|\int_{0}^{t}A^{\sigma}E(t-s)dW^{H}_{Q}(s)\right \|^{2}_{\mathbb{H}}.
\end{aligned}
\end{equation*}
Using  Lemma \ref{eqcorlem} and Cauchy-Schwarz inequality and decomposing $2\sigma=\sigma_{1}+\sigma_{2}$ imply
\begin{equation*}
\begin{aligned}
\mathbb{E}\|A^{\sigma}u\|^{2}_{\mathbb{H}}\leq&CH(2H-1)\int_{0}^{t}\int_{0}^{t}\left \langle A^{\sigma}E(t-s)Q^{1/2},A^{\sigma}E(t-r)Q^{1/2}\right \rangle|r-s|^{2H-2}drds\\
\leq &CH(2H-1)\int_{0}^{t}\left \langle A^{\sigma_{1}}E(t-s)Q^{1/2},\int_{s}^{t}A^{\sigma_{2}}E(t-r)Q^{1/2}|r-s|^{2H-2}dr\right\rangle ds\\
\leq &CH(2H-1)\int_{0}^{t}\|A^{\sigma_{1}}E(t-s)Q^{1/2}\|_{\mathcal{L}_{2}}\\
&\qquad\qquad\qquad\cdot\left \|\int_{0}^{t-s}A^{\sigma_{2}}E(t-s-r)Q^{1/2}|r|^{2H-2}dr\right \|_{\mathcal{L}_{2}}ds\\
\leq &CH(2H-1)\left (\int_{0}^{t}s^{2\alpha-1+\epsilon}\|A^{\sigma_{1}}E(s)Q^{1/2}\|^{2}_{\mathcal{L}_{2}}ds\right )^{1/2}\\
&\cdot\left (\int_{0}^{t}(t-s)^{-2\alpha+1-\epsilon}\left\|\int_{0}^{t-s}A^{\sigma_{2}}E(t-s-r)Q^{1/2}|r|^{2H-2}dr\right\|_{\mathcal{L}_{2}}^{2}ds\right )^{1/2}\\
\leq &C\uppercase\expandafter{\romannumeral1}\cdot\uppercase\expandafter{\romannumeral2}.
\end{aligned}
\end{equation*}
Combining the boundedness of $\|A^{-\rho}Q^{1/2}\|_{\mathcal{L}_{2}}=\|A^{-\rho}\|_{\mathcal{L}^{0}_{2}}$, the resolvent estimate \eqref{equresolvent}, and simple calculations yield
\begin{equation*}
\begin{aligned}
\uppercase\expandafter{\romannumeral1}^{2}\leq&C\int_{0}^{t}s^{2\alpha-1+\epsilon}\left\|\int_{\Gamma_{\theta,\kappa}}e^{zs}z^{\alpha-1}A^{\sigma_{1}}(z^{\alpha}+A)^{-1}Q^{1/2}dz\right\|^{2}_{\mathcal{L}_{2}}ds\\
\leq&C\int_{0}^{t}s^{2\alpha-1+\epsilon}\left (\int_{\Gamma_{\theta,\kappa}}|e^{zs}||z|^{\alpha-1}\left\|A^{\sigma_{1}+\rho}(z^{\alpha}+A)^{-1}\right\|\left\|A^{-\rho}\right\|_{\mathcal{L}_{2}^{0}}|dz|\right )^{2}ds\\
\leq& C\int_{0}^{t}s^{2\alpha-1+\epsilon}\left (\int_{\Gamma_{\theta,\kappa}}|e^{zs}||z|^{\alpha-1}\left\|A^{\sigma_{1}+\rho}(z^{\alpha}+A)^{-1}\right\||dz|\right )^{2}ds\\
\leq& C\int_{0}^{t}s^{2\alpha-1+\epsilon}\left (\int_{\Gamma_{\theta,\kappa}}|e^{zs}||z|^{(\sigma_{1}+\rho)\alpha-1}|dz|\right )^{2}ds\\
\leq& C\int_{0}^{t}s^{2\alpha-1+\epsilon}\left (s^{-(\rho+\sigma_{1})\alpha}\right )^{2}ds\\
\leq& C t^{2\alpha-2(\rho+\sigma_{1})\alpha+\epsilon},\\
\end{aligned}
\end{equation*}
where we need to require $\alpha-(\rho+\sigma_{1})\alpha\geq 0$ to make $\uppercase\expandafter{\romannumeral1}$ bounded and $0\leq\sigma_{1}+\rho\leq 1$ to ensure $\left\|A^{\sigma_{1}+\rho}(z^{\alpha}+A)^{-1}\right\|\leq C|z|^{(\sigma_{1}+\rho-1)\alpha}$, implying $\sigma_{1}\in [-\rho,1-\rho]$. The estimate of $\uppercase\expandafter{\romannumeral2}$ can be got similarly
\begin{equation*}
\begin{aligned}
\uppercase\expandafter{\romannumeral2}^{2}\leq&C\int_{0}^{t}(t-s)^{-2\alpha+1-\epsilon}\left\|\int_{\Gamma_{\theta,\kappa}}e^{z(t-s)}z^{\alpha-1}A^{\sigma_{2}}(z^{\alpha}+A)^{-1}Q^{1/2}z^{1-2H}dz\right\|^{2}_{\mathcal{L}_{2}}ds\\
\leq& C\int_{0}^{t}(t-s)^{-2\alpha+1-\epsilon}\left (\int_{\Gamma_{\theta,\kappa}}|e^{z(t-s)}||z|^{\alpha-2H}\left\|A^{\sigma_{2}+\rho}(z^{\alpha}+A)^{-1}\right\||dz|\right )^{2}ds\\
\leq& C\int_{0}^{t}(t-s)^{-2\alpha+1-\epsilon}\left (\int_{\Gamma_{\theta,\kappa}}|e^{z(t-s)}||z|^{(\sigma_{2}+\rho)\alpha-2H}|dz|\right )^{2}ds\\
\leq& C\int_{0}^{t}(t-s)^{-2\alpha+1-\epsilon}\left ((t-s)^{-(\rho+\sigma_{2})\alpha+2H-1}\right )^{2}ds\\
\leq& Ct^{-2\alpha-\epsilon-2(\rho+\sigma_{2})\alpha+4H}.\\
\end{aligned}
\end{equation*}
To make $\uppercase\expandafter{\romannumeral2}$ bounded and the above estimates effective, $-2\alpha-\epsilon-2(\rho+\sigma_{2})\alpha+4H>0$ and $0\leq\sigma_{2}+\rho\leq 1$ are needed, i.e., $\sigma_{2}\in [-\rho,\min(1-\rho,\frac{2H}{\alpha}-\frac{\epsilon}{2\alpha}-1-\rho)]$. Combining the estimates of $\uppercase\expandafter{\romannumeral1}$ and $\uppercase\expandafter{\romannumeral2}$, the desired result is obtained.
\end{proof}

Then we provide the H\"{o}lder regularity of the mild solution $u$.
\begin{theorem}\label{thmholder}
	Let $u$ be the mild solution of Eq. $\eqref{equretosol}$ and $\|A^{-\rho}\|_{\mathcal{L}^{0}_{2}}<\infty$ with $\rho\in [0,\frac{H}{\alpha})\cap[0,1]$. Then we have
	\begin{equation*}
	\mathbb{E}\left \|\frac{u(t)-u(t-\tau)}{\tau^{\gamma}}\right \|^{2}_{\mathbb{H}}\leq C,
	\end{equation*}
	where $\gamma< H-\rho\alpha$.
\end{theorem}
\begin{proof}
	Divide $\mathbb{E}\left\|\frac{u(t)-u(t-\tau)}{\tau^{\gamma}}\right\|_{\mathbb{H}}^{2}$ into the following two parts
	\begin{equation*}
	\begin{aligned}
	\mathbb{E}\left\|\frac{u(t)-u(t-\tau)}{\tau^{\gamma}}\right\|_{\mathbb{H}}^{2}\leq&\mathbb{E}\left\|\tau^{-\gamma}\int_{t-\tau}^{t}E(t-s)dW^{H}_{Q}\right\|_{\mathbb{H}}^{2}\\
	&+\mathbb{E}\left\|\frac{\int_{0}^{t-\tau}E(t-s)-E(t-\tau-s)dW^{H}_{Q}}{\tau^{\gamma}}\right\|_{\mathbb{H}}^{2}\\
	\leq&\uppercase\expandafter{\romannumeral1}+\uppercase\expandafter{\romannumeral2}.
	\end{aligned}
	\end{equation*}
	Using Lemma \ref{eqcorlem},  $\uppercase\expandafter{\romannumeral1}$ can be further separated into
	\begin{equation*}
	\begin{aligned}
	\uppercase\expandafter{\romannumeral1}\leq& \frac{C}{\tau^{2\gamma}}H(2H-1)\int^{t}_{t-\tau}\int^{t}_{t-\tau}\left \langle E(t-s)Q^{1/2},E(t-r)Q^{1/2}\right\rangle |r-s|^{2H-2}drds\\
	\leq &\frac{C}{\tau^{2\gamma}}H(2H-1)\left (\int_{t-\tau}^{t}(t-s)^{2\rho\alpha-1+\epsilon}\|E(t-s)Q^{1/2}\|^{2}_{\mathcal{L}_{2}}ds\right )^{1/2}\\
	&\cdot\left (\int_{t-\tau}^{t}(t-s)^{-2\rho\alpha+1-\epsilon}\left\|\int_{0}^{t-s}E(t-s-r)Q^{1/2}|r|^{2H-2}dr\right\|^{2}_{\mathcal{L}_{2}}ds\right )^{1/2}\\
	\leq &\frac{C}{\tau^{2\gamma}}\uppercase\expandafter{\romannumeral1}_{1}\cdot\uppercase\expandafter{\romannumeral1}_{2}.
	\end{aligned}
	\end{equation*}
	As for $\uppercase\expandafter{\romannumeral1}_{1}$, it follows from $\|A^{-\rho}Q^{1/2}\|_{\mathcal{L}_{2}}=\|A^{-\rho}\|_{\mathcal{L}^{0}_{2}}<\infty$, $\eqref{equresolvent}$, and the condition $\rho\in[0,1]$ that
	\begin{equation*}
	\begin{aligned}
	\uppercase\expandafter{\romannumeral1}_{1}^{2}\leq&C\int_{t-\tau}^{t}(t-s)^{2\rho\alpha-1+\epsilon}\left\|\int_{\Gamma_{\theta,\kappa}}e^{z(t-s)}z^{\alpha-1}(z^{\alpha}+A)^{-1}Q^{1/2}dz\right\|^{2}_{\mathcal{L}_{2}}ds\\
	\leq& C\int_{t-\tau}^{t}(t-s)^{2\rho\alpha-1+\epsilon}\left (\int_{\Gamma_{\theta,\kappa}}|e^{z(t-s)}||z|^{\alpha-1}\left\|A^{\rho}(z^{\alpha}+A)^{-1}\right\||dz|\right )^{2}ds\\
	\leq& C\int_{t-\tau}^{t}(t-s)^{2\rho\alpha-1+\epsilon}\left (\int_{\Gamma_{\theta,\kappa}}|e^{z(t-s)}||z|^{\rho\alpha-1}|dz|\right )^{2}ds\\
	\leq& C\int_{t-\tau}^{t}(t-s)^{2\rho\alpha-1+\epsilon}(t-s)^{-2\rho\alpha}ds\\
	\leq& C\tau^{\epsilon}.
	\end{aligned}
	\end{equation*}
Similarly, using convolution property of Laplace transform, $\uppercase\expandafter{\romannumeral1}_{2}$ satisfies
	\begin{equation*}
	\begin{aligned}
	\uppercase\expandafter{\romannumeral1}_{2}^{2}\leq&C\int_{t-\tau}^{t}(t-s)^{-2\rho\alpha+1-\epsilon}\left\|\int_{\Gamma_{\theta,\kappa}}e^{z(t-s)}z^{\alpha-1}(z^{\alpha}+A)^{-1}Q^{1/2}z^{1-2H}\right\|^{2}_{\mathcal{L}_{2}}ds\\
	\leq& C\int_{t-\tau}^{t}(t-s)^{-2\rho\alpha+1-\epsilon}\left (\int_{\Gamma_{\theta,\kappa}}|e^{z(t-s)}||z|^{\alpha-2H}\left\|A^{\rho}(z^{\alpha}+A)^{-1}\right\||dz|\right )^{2}ds\\
	\leq& C\int_{t-\tau}^{t}(t-s)^{-2\rho\alpha+1-\epsilon}\left (\int_{\Gamma_{\theta,\kappa}}|e^{z(t-s)}||z|^{\rho\alpha-2H}|dz|\right )^{2}ds\\
	\leq& C\int_{t-\tau}^{t}(t-s)^{-2\rho\alpha+1-\epsilon}\left ((t-s)^{-\rho\alpha+2H-1}\right )^{2}ds\\
	\leq& C\tau^{4H-4\rho\alpha-\epsilon}.
	\end{aligned}
	\end{equation*}
From the estimates of $\uppercase\expandafter{\romannumeral1}_{1}$ and $\uppercase\expandafter{\romannumeral1}_{2}$, it can be seen that $\gamma\leq H-\rho\alpha$ is needed to ensure the boundedness of $\uppercase\expandafter{\romannumeral1}$.
	
Following Lemma \ref{eqcorlem} and the Cauchy-Schwarz inequality, we divide $\uppercase\expandafter{\romannumeral2}$ into
\begin{equation*}
\begin{aligned}
\uppercase\expandafter{\romannumeral2}
\leq& C\tau^{-2\gamma}\int_{0}^{t-\tau}\int_{0}^{t-\tau}\left \langle ( E(t-s)-E(t-\tau-s))Q^{1/2},\right .\\
&\qquad\qquad\qquad\qquad\left .(E(t-r)-E(t-\tau-r))Q^{1/2}\right \rangle|r-s|^{2H-2}drds\\
\leq& C\tau^{-\gamma_{1}}\left( \int_{0}^{t-\tau}(t-\tau-s)^{\varrho}\| (E(t-s)-E(t-\tau-s))Q^{1/2}\|^{2}_{\mathcal{L}_{2}}ds\right)^{1/2}\\
&\cdot \tau^{-\gamma_{2}}\left (\int_{0}^{t-\tau}(t-\tau-s)^{-\varrho}\right .\\
&\qquad\qquad\qquad\left .\cdot\left\|\int_{s}^{t-\tau}(E(t-r)-E(t-\tau-r))Q^{1/2}|r-s|^{2H-2}dr\right\|_{\mathcal{L}_{2}}^{2}ds\right )^{1/2}\\
\leq &C\uppercase\expandafter{\romannumeral2}_{1}\cdot\uppercase\expandafter{\romannumeral2}_{2},
\end{aligned}
\end{equation*}
where $2\gamma=\gamma_{1}+\gamma_{2}$. With the condition $\rho\in[0,1]$ and the fact $\left |\frac{e^{z\tau}-1}{\tau^{\gamma}}\right |\leq C|z|^{\gamma}$ on $\Gamma_{\theta,\kappa}$ for $\gamma>0$ \cite{gunzburger_2018}, as for $\uppercase\expandafter{\romannumeral2}_{1}$,  by simple calculations one can get the estimate
\begin{equation*}
\begin{aligned}
&\left\|\int_{\Gamma_{\theta,\kappa}}\tau^{-\gamma_{1}}(e^{z(t-s)}-e^{z(t-\tau-s)})z^{\alpha-1}(z^{\alpha}+A)^{-1}Q^{1/2}dz\right\|^{2}_{\mathcal{L}_{2}}\\
\leq&C\left (\int_{\Gamma_{\theta,\kappa}}|e^{z(t-\tau-s)}||z|^{\gamma_{1}+\alpha-1}\left\|A^{\rho}(z^{\alpha}+A)^{-1}\right\|\left\|A^{-\rho}\right\|_{\mathcal{L}_{2}^{0}}|dz|\right )^{2}\\
\leq& C\left (\int_{\Gamma_{\theta,\kappa}}|e^{z(t-\tau-s)}||z|^{\gamma_{1}+\alpha-1}\left\|A^{\rho}(z^{\alpha}+A)^{-1}\right\||dz|\right )^{2}\\
\leq& C\left (\int_{\Gamma_{\theta,\kappa}}|e^{z(t-\tau-s)}||z|^{\alpha\rho+\gamma_{1}-1}|dz|\right )^{2},\\
\end{aligned}
\end{equation*}
which leads to
\begin{equation*}
\begin{aligned}
\uppercase\expandafter{\romannumeral2}_{1}^{2}\leq& C\int_{0}^{t-\tau}(t-\tau-s)^{\varrho-2(\alpha\rho+\gamma_{1})}ds\\
\leq& C(t-\tau)^{\varrho-2(\alpha\rho+\gamma_{1})+1},
\end{aligned}
\end{equation*}
in which $\gamma_{1}$ needs to satisfy $\varrho-2(\alpha\rho+\gamma_{1})>-1$ and $\gamma_{1}>0$, i.e., $0<\gamma_{1}<\frac{\varrho+1}{2}-\alpha\rho$, to ensure the boundedness of $\uppercase\expandafter{\romannumeral2}_{1}$.
Similarly, one has
		\begin{equation*}
	\begin{aligned}
	&\tau^{-2\gamma_{2}}\left\|\int_{0}^{t-s-\tau}(E(t-s-r)-E(t-s-\tau-r))Q^{1/2}|r|^{2H-2}dr\right\|_{\mathcal{L}_{2}}^{2}\\
	\leq& C\tau^{-2\gamma_{2}}\left \|\int_{0}^{t-s-\tau}\int_{\Gamma_{\theta,\kappa}}(e^{z(t-s-r)}-e^{z(t-s-\tau-r)})z^{\alpha-1}(z^{\alpha}+A)^{-1}Q^{1/2}dz|r|^{2H-2}dr\right\|_{\mathcal{L}_{2}}^{2}\\
	\leq&C\left (\int_{0}^{t-s-\tau}\int_{\Gamma_{\theta,\kappa}}|e^{z(t-s-\tau-r)}||z|^{\gamma_{2}+\alpha-1}\left\|A^{\rho}(z^{\alpha}+A)^{-1}\right\||dz||r|^{2H-2}dr\right)^{2}\\
	\leq&C\left (\int_{0}^{t-s-\tau}\int_{\Gamma_{\theta,\kappa}}|e^{z(t-s-\tau-r)}||z|^{\gamma_{2}+\rho\alpha-1}|dz||r|^{2H-2}dr\right)^{2}\\
	\leq& C\left (\int_{0}^{t-s-\tau}(t-s-\tau-r)^{-\rho\alpha-\gamma_{2}}|r|^{2H-2}dr\right)^{2},
	\end{aligned}
	\end{equation*}
	which implies
	\begin{equation*}
	\begin{aligned}
	\uppercase\expandafter{\romannumeral2}_{2}^{2}
	\leq& C\int_{0}^{t-\tau}(t-\tau-s)^{-\varrho}(t-\tau-s)^{4H-2-2\rho\alpha-2\gamma_{2}}ds.\\
	\end{aligned}
	\end{equation*}
To keep the boundedness of $\uppercase\expandafter{\romannumeral2}_{2}$, we need to require $\gamma_{2}>0$, $4H-2-2\rho\alpha-2\gamma_{2}-\varrho>-1$, and $-\rho\alpha-\gamma_{2}>-1$, i.e., $0<\gamma_{2}<\min(2H-\rho\alpha-\frac{\varrho+1}{2},1-\rho\alpha)$. Overall, we can take $\varrho+1=2(H-\epsilon)\geq2(\rho\alpha+\epsilon) $ and $H+\epsilon<1$ to get the desired  result.
\end{proof}

Combining Theorems \ref{thmsobo}, \ref{thmholder} and the regularity results for Eq. \eqref{equretosol1} in \cite{jin2016-1,sakamoto_2011}, we obtain
\begin{theorem}
	Let $G$ be the mild solution of Eq. \eqref{equretosol0} and $\|A^{-\rho}\|_{\mathcal{L}^{0}_{2}}<\infty$. We have
	\begin{enumerate}
		\item if $\rho<\frac{H}{\alpha}$ and $G_{0}\in \hat{H}^{q}(D)$ with $q\leq \sigma$, then
		\begin{equation*}
			(\mathbb{E}\|A^{\sigma}G\|_{\mathbb{H}}^{2})^{1/2}\leq C+Ct^{-(\sigma-q)\alpha}\|G_{0}\|_{\hat{H}^{q}(D)},
		\end{equation*}
		where $\sigma=\min(1,1-\rho,\frac{H}{\alpha}-\rho-\epsilon)$.
			\item if $\rho\in[0,\frac{H}{\alpha})\cap[0,1]$ and $G_{0}\in \mathbb{H}$, then
		\begin{equation*}
		\left (\mathbb{E}\left \|\frac{G(t)-G(t-\tau)}{\tau^{\gamma}}\right \|_{\mathbb{H}}^{2}\right )^{1/2}\leq C+Ct^{-\gamma}\|G_{0}\|_{\mathbb{H}},
		\end{equation*}
		where $\gamma<H-\rho\alpha$.
	\end{enumerate}
\end{theorem}
\section{Time discretization and error analysis}
In this section, we use backward Euler convolution quadrature \cite{lubich_1988-1,lubich_1988} to discretize the temporal operator and the corresponding error estimates are provided.

 Let the time step size $\tau=T/N$ with $N\in\mathbb{N}$, $t_i=i\tau$, $i=0,1,\ldots,N$, and $0=t_0<t_1<\cdots<t_N=T$. Denote
\begin{equation*}
	\delta_{\tau}(\zeta)=\frac{1-\zeta}{\tau}
\end{equation*}
and
\begin{equation*}
	(\delta_{\tau}(\zeta))^{\alpha}=\sum_{i=1}^{n}d^{(\alpha)}_{i}\zeta^{i}.
\end{equation*}
Introducing $u^{n}$ as the numerical solution at time $t_{n}$ and using backward Euler method to discretize the corresponding temporal operator, the semi-discrete scheme of \eqref{equretosol} can be written as
\begin{equation}\label{eqfullscheme}
	\frac{u^{n}-u^{n-1}}{\tau}+\sum_{i=0}^{n-1}d^{(1-\alpha)}_{i}Au^{n-i}=\bar{\partial}_{\tau} W^{H}_{Q}(t_{n}),
\end{equation}
where
\begin{equation}
	\bar{\partial}_{\tau} W^{H}_{Q}(t)=\left\{
	\begin{aligned}
		&0\qquad t=t_{0},\\
		&\frac{W^{H}_{Q}(t_{j})-W^{H}_{Q}(t_{j-1})}{\tau}\qquad t\in(t_{j-1},t_{j}],\\
		&0\qquad t>t_{N}.
	\end{aligned} \right.
\end{equation}

Multiplying $\zeta^{n}$ on both sides of \eqref{eqfullscheme} and summing $n$ from $1$ to $\infty$ lead to
\begin{equation*}
	\sum_{n=1}^{\infty}\frac{u^{n}-u^{n-1}}{\tau}\zeta^{n}+\sum_{n=1}^{\infty}\sum_{i=0}^{n-1}d^{(1-\alpha)}_{i}Au^{n-i}\zeta^{n}=\sum_{n=1}^{\infty}\bar{\partial}_{\tau} W^{H}_{Q}(t_{n})\zeta^{n}.
\end{equation*}
By the definition of $d^{(1-\alpha)}_{i}$, we have
\begin{equation*}
	(\delta_{\tau}(\zeta)+(\delta_{\tau}(\zeta))^{1-\alpha}A)\sum_{n=1}^{\infty}u^{n}\zeta^{n}=\sum_{n=1}^{\infty}\bar{\partial}_{\tau} W^{H}_{Q}(t_{n})\zeta^{n},
\end{equation*}
which leads to
\begin{equation*}
	\sum_{n=1}^{\infty}u^{n}\zeta^{n}=(\delta_{\tau}(\zeta))^{\alpha-1}((\delta_{\tau}(\zeta))^{\alpha}+A)^{-1}\sum_{n=1}^{\infty}\bar{\partial}_{\tau} W^{H}_{Q}(t_{n})\zeta^{n}.
\end{equation*}
Using Cauchy's integral formula and doing simple calculations lead to
\begin{equation*}
	u^{n}=\frac{\tau}{2\pi\mathbf{i}}\int_{\Gamma^{\tau}_{\theta,\kappa}}e^{zt_{n}}(\delta_{\tau}(e^{-z\tau}))^{\alpha-1}((\delta_{\tau}(e^{-z\tau}))^{\alpha}+A)^{-1}\sum_{n=1}^{\infty}\bar{\partial}_{\tau} W^{H}_{Q}(t_{n})e^{-zt_{n}}dz,
\end{equation*}
where $\Gamma^\tau_{\theta,\kappa}=\{z\in \mathbb{C}:\kappa\leq |z|\leq\frac{\pi}{\tau\sin(\theta)},|\arg z|=\theta\}\cup\{z\in \mathbb{C}:|z|=\kappa,|\arg z|\leq\theta\}$. With help of the fact \cite{gunzburger_2018}
\begin{equation*}
	\sum_{n=1}^{\infty}\bar{\partial}_{\tau} W^{H}_{Q}(t_{n})e^{-zt_{n}}=\frac{z}{e^{z\tau}-1} \widetilde{\bar{\partial}_{\tau}W^{H}_{Q}},
\end{equation*}
the solution of Eq. \eqref{eqfullscheme} has the form
\begin{equation*}
	u^{n}=\frac{1}{2\pi\mathbf{i}}\int_{\Gamma^{\tau}_{\theta,\kappa}}e^{zt_{n}}(\delta_{\tau}(e^{-z\tau}))^{\alpha-1}((\delta_{\tau}(e^{-z\tau}))^{\alpha}+A)^{-1}\frac{z\tau}{e^{z\tau}-1} \widetilde{\bar{\partial}_{\tau}W^{H}_{Q}}dz.
\end{equation*}
Introduce
\begin{equation*}
	\bar{E}(t)=\frac{1}{2\pi\mathbf{i}}\int_{\Gamma^{\tau}_{\theta,\kappa}}e^{zt}(\delta_{\tau}(e^{-z\tau}))^{\alpha-1}((\delta_{\tau}(e^{-z\tau}))^{\alpha}+A)^{-1}\frac{z\tau}{e^{z\tau}-1} dz;
\end{equation*}
the convolution property of Laplace transform gives
\begin{equation}\label{eqrepsolfull}
	u^{n}=\int_{0}^{t_{n}}\bar{E}(t_{n}-s)\bar{\partial}_{\tau}W^{H}_{Q}(s)ds.
\end{equation}
Using \eqref{eqrepsolfull} and \eqref{eqrepsol} and taking the expectation of $\|u(t_{n})-u^{n}\|^{2}_{\mathbb{H}}$ yield
\begin{equation}\label{equtimeerrorrep}
	\begin{aligned}
		&\mathbb{E}\|u(t_{n})-u^{n}\|_{\mathbb{H}}^{2}\\
		=&\mathbb{E}\left \|\int_{0}^{t_{n}}E(t_{n}-s)dW^{H}_{Q}(s)-\int_{0}^{t_{n}}\bar{E}(t_{n}-s)(\bar{\partial}_{\tau}W^{H}_{Q}(s))ds\right \|_{\mathbb{H}}^{2}\\
		\leq&\mathbb{E}\left \|\int_{0}^{t_{n}}(E(t_{n}-s)-\bar{E}(t_{n}-s))dW^{H}_{Q}(s)\right \|_{\mathbb{H}}^{2}\\
		&+\mathbb{E}\left \|\int_{0}^{t_{n}}\bar{E}(t_{n}-s)\left(dW^{H}_{Q}(s)-\bar{\partial}_{\tau}W^{H}_{Q}(s)ds \right)\right \|_{\mathbb{H}}^{2}\\
		\leq&\vartheta_{1}+\vartheta_{2}.
	\end{aligned}
\end{equation}

To estimate $\mathbb{E}\|u(t_{n})-u^{n}\|_{\mathbb{H}}^{2}$, the following lemma is needed.
\begin{lemma}[\cite{gunzburger_2018}]\label{Lemseriesest}
	Let the given $\alpha\in(0,1)$ and $\theta \in\left(\frac{\pi}{2}, \operatorname{arccot}\left(-\frac{2}{\pi}\right)\right)$,  where $arccot$ means the inverse function of $\cot$, and a fixed $\xi \in (0,1)$. Then, when $z$ lies in the region enclosed by $\Gamma^\tau_\xi=\{z=-\ln{(\xi)}/\tau+\mathbf{i}y:y\in\mathbb{R}~and~|y|\leq \pi/\tau\}$, $\Gamma^\tau_{\theta,\kappa}$, and the two lines $\mathbb{R}\pm \mathbf{i}\pi/\tau$, whenever $0<\kappa \leq \min (1 / T,-\ln (\xi) / \tau)$, $\delta_\tau(e^{-z\tau})$ and $(\delta_\tau(e^{-z\tau})+A)^{-1}$ are both analytic. Furthermore, there are the estimates
	\begin{equation*}
	\begin{aligned}
	&\delta_{\tau}\left(e^{-z \tau}\right) \in \Sigma_{\theta}&\forall z \in \Gamma_{\theta, \kappa}^{\tau},\\
	&C_{0}|z| \leq\left|\delta_{\tau}\left(e^{-z\tau }\right)\right| \leq C_{1}|z|&\forall z \in \Gamma_{\theta, \kappa}^{\tau},\\
	&\left|\delta_{\tau}\left(e^{-z\tau }\right)-z\right| \leq C \tau|z|^{2}&\forall z \in \Gamma_{\theta, \kappa}^{\tau},\\
	&\left|\delta_{\tau}\left(e^{-z\tau }\right)^{\alpha}-z^{\alpha}\right| \leq C \tau|z|^{\alpha+1}&\forall z \in \Gamma_{\theta, \kappa}^{\tau},
	\end{aligned}
	\end{equation*}
	where the constants $C_0$, $C_1$, and $C$ are independent of $\tau$ and $\kappa\in (0,\min (1 / T,-\ln (\xi) / \tau))$.
	
\end{lemma}
 In the rest of paper, we take $\kappa\leq\frac{\pi}{t_{n}|\sin(\theta)|}$.
 Then we provide the estimates of $\vartheta_{1}$ and $\vartheta_{2}$ defined in \eqref{equtimeerrorrep}.

\begin{lemma}\label{lemtime1}
	Let $\|A^{-\rho}\|_{\mathcal{L}^{0}_{2}}<\infty$ with $\rho\in [0,\frac{H}{\alpha})\cap[0,1]$. Then $\vartheta_{1}$ defined in \eqref{equtimeerrorrep} satisfies the estimate
	\begin{equation*}
		\vartheta_{1}\leq C\tau^{2H-2\rho\alpha}.
	\end{equation*}
\end{lemma}
\begin{proof}
	We split $\vartheta_{1}$ into two parts first
	\begin{equation*}
		\begin{aligned}
			\vartheta_{1}\leq& C\mathbb{E}\left \|\int_{0}^{t_{n}}\left(\int_{\Gamma_{\theta,\kappa}}e^{z(t_{n}-s)}\tilde{E}(z)dz\right.\right.\left.\left.-\int_{\Gamma^{\tau}_{\theta,\kappa}}e^{z(t_{n}-s)}\tilde{\bar{E}}(z)dz\right )dW^{H}_{Q}(s)\right \|_{\mathbb{H}}^{2}\\
			\leq& C\mathbb{E}\left \|\int_{0}^{t_{n}}\int_{\Gamma_{\theta,\kappa}\backslash\Gamma^{\tau}_{\theta,\kappa}}e^{z(t_{n}-s)}\tilde{E}(z)dzdW^{H}_{Q}(s)\right\|_{\mathbb{H}}^{2}\\
			&+C\mathbb{E}\left\|\int_{0}^{t_{n}}\int_{\Gamma^{\tau}_{\theta,\kappa}}e^{z(t_{n}-s)}\left(\tilde{E}(z)-\tilde{\bar{E}}(z)\right )dzdW^{H}_{Q}(s)\right \|_{\mathbb{H}}^{2}\\
			\leq& \vartheta_{1,1}+\vartheta_{1,2}.
		\end{aligned}	
	\end{equation*}
	According to Lemma \ref{eqcorlem}, there holds
	\begin{equation*}
		\begin{aligned}
			\vartheta_{1,1}
			\leq& CH(2H-1)\int_{0}^{t_{n}}\int_{0}^{t_{n}}\left \langle\int_{\Gamma_{\theta,\kappa}\backslash\Gamma^{\tau}_{\theta,\kappa}}e^{z(t_{n}-s)}\tilde{E}(z)dzQ^{1/2},\right .\\
			&\qquad\qquad\qquad\qquad\qquad\left. \int_{\Gamma_{\theta,\kappa}\backslash\Gamma^{\tau}_{\theta,\kappa}}e^{z(t_{n}-r)}\tilde{E}(z)dzQ^{1/2}\right \rangle |r-s|^{2H-2}dsdr\\
			\leq& CH(2H-1)\int_{0}^{t_{n}}\left \langle\int_{\Gamma_{\theta,\kappa}\backslash\Gamma^{\tau}_{\theta,\kappa}}e^{z(t_{n}-s)}\tilde{E}(z)dzQ^{1/2},\right .\\
			&\qquad\qquad\qquad\left. \int_{s}^{t_{n}}\int_{\Gamma_{\theta,\kappa}\backslash\Gamma^{\tau}_{\theta,\kappa}}e^{z(t_{n}-r)}\tilde{E}(z)dzQ^{1/2}|r-s|^{2H-2}dr\right \rangle ds\\
			\leq& C\left (\int_{0}^{t_{n}}(t_{n}-s)^{\varrho_{1}}\left \|\int_{\Gamma_{\theta,\kappa}\backslash\Gamma^{\tau}_{\theta,\kappa}}e^{z(t_{n}-s)}\tilde{E}(z)dzQ^{1/2}\right \|_{\mathcal{L}_{2}}^{2}ds\right )^{1/2}\\
			\cdot&\left (\int_{0}^{t_{n}}(t_{n}-s)^{-\varrho_{1}}\left\| \int_{s}^{t_{n}}\int_{\Gamma_{\theta,\kappa}\backslash\Gamma^{\tau}_{\theta,\kappa}}e^{z(t_{n}-r)}\tilde{E}(z)dzQ^{1/2}|r-s|^{2H-2}dr\right \|_{\mathcal{L}_{2}}^{2} ds\right )^{1/2}\\
			\leq& C\vartheta_{1,1,1}\cdot\vartheta_{1,1,2}.
		\end{aligned}
	\end{equation*}
	Consider $\vartheta_{1,1,1}$ first. Combining the boundedness of $\|A^{\rho}\|_{\mathcal{L}^{0}_{2}}$, the condition $\rho\in[0,1]$, the resolvent estimate \eqref{equresolvent}, and the Cauchy-Schwartz inequality gives
	\begin{equation*}
		\begin{aligned}
			\vartheta_{1,1,1}^{2}\leq& C\int_{0}^{t_{n}}(t_{n}-s)^{\varrho_{1}}\left (\int_{\Gamma_{\theta,\kappa}\backslash\Gamma^{\tau}_{\theta,\kappa}}|e^{z(t_{n}-s)}||z|^{\alpha-1}\left \|(z^{\alpha}+A)^{-1}Q^{1/2}\right \|_{\mathcal{L}_{2}}|dz|\right )^{2}ds\\
			\leq& C\int_{0}^{t_{n}}(t_{n}-s)^{\varrho_{1}}\left (\int_{\Gamma_{\theta,\kappa}\backslash\Gamma^{\tau}_{\theta,\kappa}}|e^{z(t_{n}-s)}||z|^{\alpha-1}\left \|(z^{\alpha}+A)^{-1}A^{\rho}\right \||dz|\right )^{2}ds\\
			\leq& C\int_{0}^{t_{n}}(t_{n}-s)^{\varrho_{1}}\left (\int_{\Gamma_{\theta,\kappa}\backslash\Gamma^{\tau}_{\theta,\kappa}}|e^{z(t_{n}-s)}||z|^{\rho\alpha-1}|dz|\right )^{2}ds\\
			\leq& C\int_{0}^{t_{n}}(t_{n}-s)^{\varrho_{1}}\int_{\Gamma_{\theta,\kappa}\backslash\Gamma^{\tau}_{\theta,\kappa}}|e^{2z(t_{n}-s)}||z|^{\varrho_{1}-\epsilon}|dz|\int_{\Gamma_{\theta,\kappa}\backslash\Gamma^{\tau}_{\theta,\kappa}}|z|^{-\varrho_{1}+2\rho\alpha-2+\epsilon}|dz|ds.\\
		\end{aligned}
	\end{equation*}
Taking $-\varrho_{1}+2\rho\alpha-2+\epsilon<-1$, i.e., $\varrho_{1}>2\rho\alpha-1+\epsilon$ and using the definitions of $\Gamma_{\theta,\kappa}$ and $\Gamma_{\theta,\kappa}^{\tau}$ lead to
	\begin{equation*}
		\vartheta_{1,1,1}^{2}\leq C\tau^{\varrho_{1}-2\rho\alpha+1-\epsilon}\int_{\Gamma_{\theta, \kappa}\backslash\Gamma^{\tau}_{\theta, \kappa}}|z|^{-1-\epsilon}ds\leq C\tau^{\varrho_{1}-2\rho\alpha+1} .
	\end{equation*}
	
	Similarly, for $\vartheta_{1,1,2}$, with the assumptions $\rho\in[0,1]$ and $2\rho\alpha-4H+\varrho_{1}+\epsilon<-1$, i.e., $\varrho_{1}<4H-2\rho\alpha-1-\epsilon$, there holds
	
	\begin{equation*}
		\begin{aligned}
			&\vartheta_{1,1,2}^{2}\\
			\leq& C\int_{0}^{t_{n}}(t_{n}-s)^{-\varrho_{1}}\left (\int_{\Gamma_{\theta,\kappa}\backslash\Gamma^{\tau}_{\theta,\kappa}}|e^{z(t_{n}-s)}||z|^{\alpha-2H}\left \|(z^{\alpha}+A)^{-1}Q^{1/2}\right \|_{\mathcal{L}_{2}}|dz|\right )^{2}ds\\
			\leq& C\int_{0}^{t_{n}}(t_{n}-s)^{-\varrho_{1}}\left (\int_{\Gamma_{\theta,\kappa}\backslash\Gamma^{\tau}_{\theta,\kappa}}|e^{z(t_{n}-s)}||z|^{\alpha-2H}\left \|(z^{\alpha}+A)^{-1}A^{\rho}\right \||dz|\right )^{2}ds\\
			\leq& C\int_{0}^{t_{n}}(t_{n}-s)^{-\varrho_{1}}\left (\int_{\Gamma_{\theta,\kappa}\backslash\Gamma^{\tau}_{\theta,\kappa}}|e^{z(t_{n}-s)}||z|^{\rho\alpha-2H}|dz|\right )^{2}ds\\
			\leq& C\int_{0}^{t_{n}}(t_{n}-s)^{-\varrho_{1}}\int_{\Gamma_{\theta,\kappa}\backslash\Gamma^{\tau}_{\theta,\kappa}}|e^{2z(t_{n}-s)}||z|^{-\varrho_{1}-\epsilon}|dz|ds\\
			&\quad\cdot\int_{\Gamma_{\theta,\kappa}\backslash\Gamma^{\tau}_{\theta,\kappa}}|z|^{2\rho\alpha-4H+\varrho_{1}+\epsilon}|dz|\\
			\leq& C\tau^{4H-2\rho\alpha-\varrho_{1}-1}.\\
		\end{aligned}
	\end{equation*}
Since $4H-2\rho\alpha-1-\epsilon>2\rho\alpha-1+\epsilon$ holds directly, thus
	\begin{equation*}
		\vartheta_{1,1}\leq C\tau^{2H-2\rho\alpha}.
	\end{equation*}

	To estimate $\vartheta_{1,2}$, we introduce
	\begin{equation*}
		\mathcal{E}_{\tau}(z)=z^{\alpha-1}(z^{\alpha}+A)^{-1}-(\delta_{\tau}(e^{-z\tau}))^{\alpha-1}((\delta_{\tau}(e^{-z\tau}))^{\alpha}+A)^{-1}\frac{z\tau}{e^{z\tau}-1}.
	\end{equation*}
	Similar to $\vartheta_{1,1}$, we also spit $\vartheta_{1,2}$ into two parts
	\begin{equation*}
		\begin{aligned}
			\vartheta_{1,2}\leq& C\left (\int_{0}^{t_{n}}(t_{n}-s)^{\varrho_{2}}\left \|\int_{\Gamma^{\tau}_{\theta,\kappa}}e^{z(t_{n}-s)}\mathcal{E}_{\tau}(z)Q^{1/2}dz\right \|_{\mathcal{L}_{2}}^{2}ds\right )^{1/2}\\
			&\cdot\left (\int_{0}^{t_{n}}(t_{n}-s)^{-\varrho_{2}}\left\| \int_{s}^{t_{n}}\int_{\Gamma^{\tau}_{\theta,\kappa}}e^{z(t_{n}-r)}\mathcal{E}_{\tau}(z)Q^{1/2}dz|r-s|^{2H-2}dr\right \|_{\mathcal{L}_{2}}^{2} ds\right )^{1/2}\\
			\leq& C\vartheta_{1,2,1}\cdot\vartheta_{1,2,2}.
		\end{aligned}
	\end{equation*}
For $\vartheta_{1,2,1}$, by the assumption $\|A^{-\rho}\|_{\mathcal{L}_{2}^{0}}<\infty$, we have
	\begin{equation*}
		\begin{aligned}
			\vartheta_{1,2,1}^{2}\leq& C\int_{0}^{t_{n}}(t_{n}-s)^{\varrho_{2}}\left (\int_{\Gamma^{\tau}_{\theta,\kappa}}|e^{z(t_{n}-s)}|\left \|\mathcal{E}_{\tau}(z)Q^{1/2}\right\|_{\mathcal{L}_{2}}|dz|\right )^{2} ds\\
			\leq& C\int_{0}^{t_{n}}(t_{n}-s)^{\varrho_{2}}\left (\int_{\Gamma^{\tau}_{\theta,\kappa}}|e^{z(t_{n}-s)}|\left \|\mathcal{E}_{\tau}(z)A^{\rho}\right\||dz|\right )^{2} ds.
		\end{aligned}
	\end{equation*}
	Here we first provide the  estimate 
	\begin{equation}\label{eqerrorH2l2}
		\begin{aligned}
			\left \|\mathcal{E}_{\tau}(z)A\right\|\leq& \left \|\left(z^{\alpha-1}(z^{\alpha}+A)^{-1}-(\delta_{\tau}(e^{-z\tau}))^{\alpha-1}((\delta_{\tau}(e^{-z\tau}))^{\alpha}+A)^{-1}\frac{z\tau}{e^{z\tau}-1}\right )A\right  \|\\			
			\leq &\left \|\left(z^{\alpha-1}(z^{\alpha}+A)^{-1}\frac{e^{z\tau}-1-z\tau}{e^{z\tau}-1}\right )A\right  \|\\
			&+\left \|\left(z^{\alpha-1}(z^{\alpha}+A)^{-1}-(\delta_{\tau}(e^{-z\tau}))^{\alpha-1}((\delta_{\tau}(e^{-z\tau}))^{\alpha}+A)^{-1}\right )\frac{z\tau}{e^{z\tau}-1}A\right  \|\\
			\leq& C\tau |z|^{\alpha},
		\end{aligned}
	\end{equation}
	which can be got by Lemma \ref{Lemseriesest} and the Taylor expansion. Similarly, we have
	\begin{equation}\label{eqerrorl2l2}
		\left \|\mathcal{E}_{\tau}(z)\right\|\leq C\tau.
	\end{equation}
	Thus according to Eqs. \eqref{eqerrorH2l2}, \eqref{eqerrorl2l2}, and the interpolation property, there holds
	\begin{equation}\label{eqerrorintres}
		\|\mathcal{E}_{\tau}(z)A^{\beta}\|\leq C\tau|z|^{\beta\alpha},
	\end{equation}
	where $\beta\in[0,1]$.
	 Assuming $\rho\in[0,1]$, $\varrho_{2}>-1$, and $2\rho\alpha-\varrho_{2}-\epsilon>-1$, i.e., $\varrho_{2}<2\rho\alpha+1+\epsilon$ and using \eqref{eqerrorintres} and $\kappa\leq\frac{\pi}{t_{n}|\sin(\theta)|} $ lead to
	\begin{equation*}
		\begin{aligned}
			\vartheta_{1,2,1}^{2}\leq&C\tau^{2}\int_{0}^{t_{n}}(t_{n}-s)^{\varrho_{2}}\left (\int_{\Gamma^{\tau}_{\theta,\kappa}}|e^{z(t_{n}-s)}||z|^{\rho\alpha}|dz|\right )^{2} ds\\
			\leq&C\tau^{2}\int_{0}^{t_{n}}(t_{n}-s)^{\varrho_{2}}\int_{\Gamma^{\tau}_{\theta,\kappa}}|e^{2z(t_{n}-s)}||z|^{\varrho_{2}+\epsilon}|dz| \int_{\Gamma^{\tau}_{\theta,\kappa}}|z|^{2\rho\alpha-\varrho_{2}-\epsilon}|dz|ds\\
			\leq&C\tau^{2}\int_{\Gamma^{\tau}_{\theta,\kappa}}|z|^{-1+\epsilon}|dz| \int_{\Gamma^{\tau}_{\theta,\kappa}}|z|^{2\rho\alpha-\varrho_{2}-\epsilon}|dz|\\
			\leq& C\tau^{1+\varrho_{2}-2\rho\alpha}.
		\end{aligned}
	\end{equation*}
	
	As for $\vartheta_{1,2,2}$, using the fact $\kappa\leq\frac{\pi}{t_{n}|\sin(\theta)|}$, there holds
	\begin{equation*}
		\begin{aligned}
			\vartheta_{1,2,2}^{2}\leq& C\int_{0}^{t_{n}}(t_{n}-s)^{-\varrho_{2}}\left (\int_{\Gamma^{\tau}_{\theta,\kappa}}|e^{z(t_{n}-s)}|\left \|\mathcal{E}_{\tau}(z)Q^{1/2}\right\|_{\mathcal{L}_{2}}|z|^{1-2H}|dz|\right )^{2} ds\\
			\leq& C\int_{0}^{t_{n}}(t_{n}-s)^{-\varrho_{2}}\left (\int_{\Gamma^{\tau}_{\theta,\kappa}}|e^{z(t_{n}-s)}|\left \|\mathcal{E}_{\tau}(z)A^{\rho}\right\||z|^{1-2H}|dz|\right )^{2} ds\\
			\leq&C\tau^{2}\int_{0}^{t_{n}}(t_{n}-s)^{-\varrho_{2}}\left (\int_{\Gamma^{\tau}_{\theta,\kappa}}|e^{z(t_{n}-s)}||z|^{1-2H+\rho\alpha}|dz|\right )^{2} ds\\
			\leq&C\tau^{2}\int_{0}^{t_{n}}(t_{n}-s)^{-\varrho_{2}}\int_{\Gamma^{\tau}_{\theta,\kappa}}|e^{2z(t_{n}-s)}||z|^{-\varrho_{2}+\epsilon}|dz| \int_{\Gamma^{\tau}_{\theta,\kappa}}|z|^{2-4H+2\rho\alpha+\varrho_{2}-\epsilon}|dz|ds\\
			\leq&C\tau^{4H-2\rho\alpha-1-\varrho_{2}},\\
		\end{aligned}
	\end{equation*}
	under the assumptions $\rho\in[0,1]$,  $\varrho_{2}<1$, and $2-4H+2\rho\alpha+\varrho_{2}-\epsilon>-1$, i.e., $\varrho_{2}>4H-3-2\rho\alpha-\epsilon$. Since $H\in(1/2,1)$ and $\rho\alpha\in(0,H)$, then $4H-3-2\rho\alpha-\epsilon<-1<1<2\rho\alpha+1-\epsilon$. Choosing a suitable $\varrho_{2}$ results in
	\begin{equation*}
		\vartheta_{1,2}\leq C\tau^{2H-2\rho\alpha}.
	\end{equation*}
Combining the above estimates of $\vartheta_{1,1}$ and $\vartheta_{1,2}$, the desired result is obtained.
\end{proof}

Further we provide the estimate of $\vartheta_{2}$.
\begin{lemma}\label{lemtime2}
	Let $\|A^{-\rho}\|_{\mathcal{L}^{0}_{2}}<\infty$ with $\rho\in [0,\frac{H}{\alpha})\cap[0,1]$.  Then $\vartheta_{2}$ defined in \eqref{equtimeerrorrep} satisfies the estimate
	\begin{equation*}
		\vartheta_{2}\leq C\tau^{2H-2\rho\alpha}.
	\end{equation*}
\end{lemma}
\begin{proof}
	For $\vartheta_{2}$, simple calculations lead to
	\begin{equation*}
		\begin{aligned}
			\vartheta_{2}\leq &C \mathbb{E}\left \|\sum_{i=1}^{n}\int_{t_{i-1}}^{t_{i}}\bar{E}(t_{n}-s)\left (dW^{H}_{Q}(s)-\bar{\partial}_{\tau}W^{H}_{Q}(s)ds\right )\right \|_{\mathbb{H}}^{2}\\
			\leq &C \mathbb{E}\left \|\sum_{i=1}^{n}\int_{t_{i-1}}^{t_{i}}\frac{1}{\tau}\int_{t_{i-1}}^{t_{i}}(\bar{E}(t_{n}-s)-\bar{E}(t_{n}-\xi))d\xi dW^{H}_{Q}(s)\right \|_{\mathbb{H}}^{2}\\
			\leq &C \mathbb{E}\left \|\int_{0}^{t_{n}}\frac{1}{\tau}\sum_{i=1}^{n}\chi_{(t_{i-1},t_{i}]}(s)\int_{t_{i-1}}^{t_{i}}(\bar{E}(t_{n}-s)-\bar{E}(t_{n}-\xi))d\xi dW^{H}_{Q}(s)\right \|_{\mathbb{H}}^{2},
		\end{aligned}
	\end{equation*}
	where $\chi_{(a,b]}(x)$ means  the characteristic function on $(a,b]$.
	Using Lemma \ref{eqcorlem}, we can split $\vartheta_{2}$ into two parts
	\begin{equation*}
		\begin{aligned}
			&\vartheta_{2}\\
			\leq& \int_{0}^{t_{n}}\int_{0}^{t_{n}}\left \langle\frac{1}{\tau}\sum_{i=1}^{n}\chi_{(t_{i-1},t_{i}]}(s)\int_{t_{i-1}}^{t_{i}}(\bar{E}(t_{n}-s)-\bar{E}(t_{n}-\xi))Q^{1/2}d\xi,\right .\\
			&\left .\frac{1}{\tau}\sum_{i=1}^{n}\chi_{(t_{i-1},t_{i}]}(r)\int_{t_{i-1}}^{t_{i}}(\bar{E}(t_{n}-r)-\bar{E}(t_{n}-\xi))Q^{1/2}d\xi\right \rangle|r-s|^{2H-2} dr ds\\
			\leq& \int_{0}^{t_{n}}\left \langle\frac{1}{\tau}\sum_{i=1}^{n}\chi_{(t_{i-1},t_{i}]}(s)\int_{t_{i-1}}^{t_{i}}(\bar{E}(t_{n}-s)-\bar{E}(t_{n}-\xi))Q^{1/2}d\xi,\right .\\
			&\left .\int_{s}^{t_{n}}\frac{1}{\tau}\sum_{i=1}^{n}\chi_{(t_{i-1},t_{i}]}(r)\int_{t_{i-1}}^{t_{i}}(\bar{E}(t_{n}-r)-\bar{E}(t_{n}-\xi))Q^{1/2}d\xi |r-s|^{2H-2}dr\right \rangle  ds\\
			\leq& \left (\int_{0}^{t_{n}}(t_{n}-s)^{\varrho_{3}}\left \|\frac{1}{\tau}\sum_{i=1}^{n}\chi_{(t_{i-1},t_{i}]}(s)\int_{t_{i-1}}^{t_{i}}(\bar{E}(t_{n}-s)-\bar{E}(t_{n}-\xi))Q^{1/2}d\xi\right \|_{\mathcal{L}_{2}}^{2}ds\right )^{1/2}\\
			&\cdot \left (\int_{0}^{t_{n}}(t_{n}-s)^{-\varrho_{3}}\left\| \int_{s}^{t_{n}}\frac{1}{\tau}\sum_{i=1}^{n}\chi_{(t_{i-1},t_{i}]}(r)\right. \right.\\
			&\left .\left .\qquad\qquad\cdot \int_{t_{i-1}}^{t_{i}}(\bar{E}(t_{n}-r)-\bar{E}(t_{n}-\xi))Q^{1/2}d\xi |r-s|^{2H-2}dr\right \|_{\mathcal{L}_{2}}^{2} ds\right )^{1/2}\\
			\leq& C\vartheta_{2,1}\cdot\vartheta_{2,2}.
		\end{aligned}
	\end{equation*}
	
	For $\vartheta_{2,1}$, using the resolvent estimate \eqref{equresolvent} and Lemma \ref{Lemseriesest} gives
	\begin{equation*}
		\begin{aligned}
			&\left(\frac{1}{\tau}\sum_{i=1}^{n}\chi_{(t_{i-1},t_{i}]}(s)\int_{t_{i-1}}^{t_{i}}\left \|(\bar{E}(t_{n}-s)-\bar{E}(t_{n}-\xi))Q^{1/2}\right\|_{\mathcal{L}_{2}}d\xi\right)^{2} \\
			\leq&C\left (\frac{1}{\tau}\sum_{i=1}^{n}\chi_{(t_{i-1},t_{i}]}(s)\int_{t_{i-1}}^{t_{i}}\left \|\int_{\Gamma^{\tau}_{\theta,\kappa}}(e^{z(t_{n}-s)}-e^{z(t_{n}-\xi)})(\delta_{\tau}(e^{-z\tau}))^{\alpha-1}\right.\right.\\
			&\qquad\qquad\qquad\qquad\left .\left .\cdot((\delta_{\tau}(e^{-z\tau}))^{\alpha}+A)^{-1}\frac{z\tau}{e^{z\tau}-1}A^{\rho}dz\right\|d\xi \right)^{2}\\
			\leq&C\tau^{2}\left(\int_{\Gamma^{\tau}_{\theta,\kappa}}|e^{z(t_{n}-s)}||z|^{\rho\alpha}|dz|\right)^{2} .\\
			\end{aligned}
			\end{equation*}
			Thus
			\begin{equation*}
			\begin{aligned}
			\vartheta_{2,1}^{2}\leq&C\tau^{2}\int_{0}^{t_{n}}(t_{n}-s)^{\varrho_{3}}\int_{\Gamma^{\tau}_{\theta,\kappa}}|e^{2z(t_{n}-s)}||z|^{\varrho_{3}-\epsilon}|dz|\int_{\Gamma^{\tau}_{\theta,\kappa}}|z|^{-\varrho_{3}+2\rho\alpha+\epsilon}|dz| ds\\
			\leq& C\tau^{1+\rho_{3}-2\rho\alpha-\epsilon},
		\end{aligned}
	\end{equation*}
	where we use $\kappa\leq\frac{\pi}{t_{n}|\sin(\theta)|}$ and take $-\rho_{3}+2\rho\alpha+\epsilon>-1$, i.e., $\rho_{3}<2\rho\alpha+\epsilon+1$.

	Similar to the estimate of $\vartheta_{2,1}$,
	we have

	\begin{equation*}
		\begin{aligned}
			&\left\| \int_{s}^{t_{n}}\frac{1}{\tau}\sum_{i=1}^{n}\chi_{(t_{i-1},t_{i}]}(r)\int_{t_{i-1}}^{t_{i}}(\bar{E}(t_{n}-r)-\bar{E}(t_{n}-\xi))Q^{1/2}d\xi |r-s|^{2H-2}dr\right \|_{\mathcal{L}_{2}}^{2}\\
			\leq &C\left( \int_{s}^{t_{n}}\frac{1}{\tau}\sum_{i=1}^{n}\chi_{(t_{i-1},t_{i}]}(r)\int_{t_{i-1}}^{t_{i}}\left\|(\bar{E}(t_{n}-r)-\bar{E}(t_{n}-\xi))Q^{1/2}\right\|_{\mathcal{L}_{2}}d\xi |r-s|^{2H-2}dr\right )^{2}\\
			\leq &C\left( \int_{s}^{t_{n}}\frac{1}{\tau}\sum_{i=1}^{n}\chi_{(t_{i-1},t_{i}]}(r)\int_{t_{i-1}}^{t_{i}}\int_{\Gamma^{\tau}_{\theta,\kappa}}\left |e^{z(t_{n}-r)}-e^{z(t_{n}-\xi)}\right |\right .\\
			&\left .\qquad\cdot\left\|((\delta_{\tau}(e^{-z\tau}))^{\alpha-1}(((\delta_{\tau}(e^{-z\tau}))^{\alpha}+A)^{-1}Q^{1/2}\right\|_{\mathcal{L}_{2}}d\xi |r-s|^{2H-2}dr\right )^{2}\\
			\leq &C\tau^{2}\left( \int_{0}^{t_{n}-s}\int_{\Gamma^{\tau}_{\theta,\kappa}}\left |e^{z(t_{n}-s-r)}\right ||z|^{\rho\alpha}|dz| |r|^{2H-2}dr\right )^{2}\\
			\leq &C\tau^{2} \int_{0}^{t_{n}-s}\left(\int_{\Gamma^{\tau}_{\theta,\kappa}}\left |e^{z(t_{n}-s-r)}\right ||z|^{\rho\alpha}|dz|\right )^{2}|r|^{-1+\epsilon}dr\int_{0}^{t_{n}-s} |r|^{4H-3-\epsilon}dr\\
			\leq &C\tau^{2}(t_{n}-s)^{4H-2-\epsilon}\\
			&\qquad\qquad\cdot \int_{0}^{t_{n}-s}\int_{\Gamma^{\tau}_{\theta,\kappa}}\left |e^{2z(t_{n}-s-r)}\right ||z|^{1+\sigma+2\rho\alpha}|dz|\int_{\Gamma^{\tau}_{\theta,\kappa}}|z|^{-1-\sigma}|dz||r|^{-1+\epsilon}dr\\
			\leq &C\tau^{2+\sigma}(t_{n}-s)^{4H-2-\epsilon} \int_{0}^{t_{n}-s}\int_{\Gamma^{\tau}_{\theta,\kappa}}\left |e^{2z(t_{n}-s-r)}\right ||z|^{1+\sigma+2\rho\alpha-\epsilon}|dz||r|^{-1+\epsilon}dr\\
			\leq &C\tau^{2+\sigma}(t_{n}-s)^{4H-2-2-\sigma-2\rho\alpha+\epsilon}, \\
		\end{aligned}
	\end{equation*}
	where we use the assumption that $\sigma<0$ and $1+\sigma+2\rho\alpha-\epsilon<0$, i.e., $\sigma<-1-2\rho\alpha+\epsilon$. To preserve the boundedness of $\vartheta_{2,2}$, $\sigma$ should satisfy $-\varrho_{3}+4H-4-\sigma-2\rho\alpha+\epsilon>-1$, i.e., $\sigma<4H-3-2\rho\alpha-\varrho_{3}+\epsilon$. Thus the desired result is obtained.
	\end{proof}

Combining Lemmas \ref{lemtime1} and \ref{lemtime2}, we have
\begin{theorem}
	Let $u(t_{n})$ and $u^{n}$ be the solutions of Eqs. \eqref{equretosol} and \eqref{eqfullscheme} and $\|A^{-\rho}\|_{\mathcal{L}^{0}_{2}}<\infty$ with $\rho\in[0,\frac{H}{\alpha})\cap[0,1]$. Then there holds 
	\begin{equation*}
		\left (\mathbb{E}\|u(t_{n})-u^{n}\|_{\mathbb{H}}^{2}\right )^{1/2}\leq C\tau^{H-\rho\alpha}.
	\end{equation*}
\end{theorem}

On the other hand, we use the same way to discretize Eq. \eqref{equretosol1} and get
\begin{equation}\label{eqfullscheme1}
\frac{v^{n}-v^{n-1}}{\tau}+\sum_{i=0}^{n-1}d^{(1-\alpha)}_{i}Av^{n-i}=0
\end{equation}
with $v^{0}=G_{0}$. Combining the results provided in \cite{jin2016-1}, we have the estimate of $G$.
\begin{theorem}\label{thmGtime}
	Let $G$ be the solution of Eq. \eqref{equretosol0} and $G^{n}=u^{n}+v^{n}$, where $u^{n}$ and $v^{n}$ are the solutions of Eqs. \eqref{eqfullscheme} and \eqref{eqfullscheme1}, respectively. Let $\|A^{-\rho}\|_{\mathcal{L}^{0}_{2}}<\infty$ with $\rho\in[0,\frac{H}{\alpha})\cap[0,1]$. Assume $G_{0}\in\hat{H}^{q} (D) $ with $q\leq 2$. There holds
	\begin{equation*}
		(\mathbb{E}\|G(t_{n})-G^{n}\|_{\mathbb{H}}^{2})^{1/2}\leq C\tau^{H-\rho\alpha }+Ct^{q\alpha/2-1}\tau\|G_{0}\|_{\hat{H}^{q} (D)}.
	\end{equation*}
\end{theorem}
\section{Spatial discretization and error analysis}
In this section, we discretize Laplace operator by the finite element method and provide the error estimates for the fully discrete scheme of Eq. \eqref{equretosol}. Let $\mathcal{T}_h$ be a shape regular quasi-uniform partition of the domain $D$, where $h$ is the maximum diameter. Denote $ X_h $ as piecewise linear finite element space
\begin{equation*}
X_{h}=\{\nu_h\in C(\bar{D}): \nu_h|_\mathbf{T}\in \mathcal{P}^1,\  \forall \mathbf{T}\in\mathcal{T}_h,\ \nu_h|_{\partial D}=0\},
\end{equation*}
where $\mathcal{P}^1$ denotes the set of piecewise polynomials of degree $1$ over $\mathcal{T}_h$. Then, we denote $(\cdot,\cdot)$ as the $L^{2}$ inner product and introduce the $ L^2 $-orthogonal projection $ P_h:\ L^2(D)\rightarrow X_h $ and the Ritz projection $ R_h:\ H^1_0(D)\rightarrow X_h $ \cite{Thomee_2006}, respectively, by
\begin{equation*}
\begin{aligned}
&(P_hu,\nu_h)=(u,\nu_h) \ \quad\forall \nu_h\in X_h,\\
&(\nabla R_h u,\nabla \nu_h)=(\nabla u, \nabla \nu_h) \ \quad\forall \nu_h\in X_h.
\end{aligned}
\end{equation*}
Define $A_h$ by $(A_hu_{h},\nu_{h})=(\nabla u_{h}, \nabla \nu_{h})$ with $u_{h},\nu_{h}\in X_{h}$. The fully discrete Galerkin scheme for Eq. \eqref{equretosol} reads: For every $t\in (0,T]$, find $ u^{n}_{h}\in X_h$ such that
\begin{equation}\label{eqsemischeme}
\left \{
\begin{aligned}
&\left(\frac{u^{n}_h-u^{n-1}_{h}}{\tau},\nu_{h}\right)+\sum_{i=0}^{n-1}d^{(1-\alpha)}_{i}(A_hu^{n-i}_h,\nu_{h})=\\
&\qquad\qquad\left (\frac{W^{H}_{Q}(t_{n})-W^{H}_{Q}(t_{n-1})}{\tau},\nu_{h}\right )\quad \forall \nu_{h}\in X_h,
\\
&u^{0}_h=0. 
\\
\end{aligned}
\right.
\end{equation}
Equation \eqref{eqsemischeme} can also  be written as
\begin{equation}\label{eqsemischemeop}
	\frac{u^{n}_h-u^{n-1}_{h}}{\tau}+\sum_{i=0}^{n-1}d^{(1-\alpha)}_{i}A_hu^{n-i}_h=P_{h}\frac{W^{H}_{Q}(t_{n})-W^{H}_{Q}(t_{n-1})}{\tau}.
\end{equation}
Introduce 
\begin{equation*}
	\bar{E}_{h}(t)=\frac{1}{2\pi\mathbf{i}}\int_{\Gamma^{\tau}_{\theta,\kappa}}e^{zt}(\delta_{\tau}(e^{-z\tau}))^{\alpha-1}((\delta_{\tau}(e^{-z\tau}))^{\alpha}+A_{h})^{-1}\frac{z\tau}{e^{z\tau}-1}dz.
\end{equation*}
Multiplying $\zeta^{n}$ on both sides of \eqref{eqsemischemeop} and summing $n$ from $1$ to $\infty$ lead to that the representation of the solution of \eqref{eqsemischeme} is
\begin{equation}\label{eqrepsolnum}
	\begin{aligned}
		u^{n}_{h}=&\int_{0}^{t_{n}}\bar{E}_{h}(t_{n}-s)P_{h}\bar{\partial}_{\tau}W^{H}_{Q}(s)ds\\
		=&\frac{1}{\tau}\sum_{i=1}^{n}\int_{t_{i-1}}^{t_{i}}\bar{E}_{h}(t_{n}-\xi)P_{h}\int_{t_{i-1}}^{t_{i}}dW^{H}_{Q}(s)d\xi\\
		=&\int_{0}^{t_{n}}\frac{1}{\tau}\sum_{i=1}^{n}\chi_{(t_{i-1},t_{i}]}(s)\int_{t_{i-1}}^{t_{i}}\bar{E}_{h}(t_{n}-\xi)P_{h}d\xi dW^{H}_{Q}(s).
	\end{aligned}
\end{equation}
Similarly, the solution of \eqref{eqfullscheme} can be represented
as
\begin{equation}\label{eqrepsolfull2}
	u^{n}=\int_{0}^{t_{n}}\frac{1}{\tau}\sum_{i=1}^{n}\chi_{(t_{i-1},t_{i}]}(s)\int_{t_{i-1}}^{t_{i}}\bar{E}(t_{n}-\xi)d\xi dW^{H}_{Q}(s).
\end{equation}
Denoting $\mathcal{E}_{h}(t)=\bar{E}(t)-\bar{E}_{h}(t)P_{h}$ and  taking the expectation of $\|u^{n}-u^{n}_{h}\|^{2}_{\mathbb{H}}$ imply
	\begin{equation}\label{eqspterrep0}
		\begin{aligned}
			\mathbb{E}\|u^{n}-u^{n}_{h}\|_{\mathbb{H}}^{2}=&\mathbb{E}\left \|\int_{0}^{t_{n}}\frac{1}{\tau}\sum_{i=1}^{n}\chi_{(t_{i-1},t_{i}]}(s)\int_{t_{i-1}}^{t_{i}}\bar{E}(t_{n}-\xi)-\bar{E}_{h}(t_{n}-\xi)P_{h}d\xi dW^{H}_{Q}(s)\right \|_{\mathbb{H}}^{2}\\
			=&\mathbb{E}\left \|\int_{0}^{t_{n}}\frac{1}{\tau}\sum_{i=1}^{n}\chi_{(t_{i-1},t_{i}]}(s)\int_{t_{i-1}}^{t_{i}}\mathcal{E}_{h}(t_{n}-\xi)d\xi dW^{H}_{Q}(s)\right \|_{\mathbb{H}}^{2}.
		\end{aligned}
	\end{equation}
	To get the estimate of $	\mathbb{E}\|u-u_{h}\|_{\mathbb{H}}^{2}$, we consider some estimates of $\mathcal{E}_{h}(t)$ first. From \cite{jin2019-1,Thomee_2006}, it holds 
\begin{equation}\label{eql2l21}
	\|\tilde{\mathcal{E}}_{h}(z)\|\leq C|z|^{\alpha-1}h^{2}\quad {\rm for}~ z\in \Gamma^{\tau}_{\theta, \kappa}.
\end{equation}
Using the resolvent estimate \eqref{equresolvent}, the following estimate also holds
\begin{equation}\label{eql2l22}
	\|\tilde{\mathcal{E}}_{h}(z)\|\leq C|z|^{-1}\quad {\rm for}~ z\in \Gamma^{\tau}_{\theta, \kappa}.
\end{equation}
From \cite{jin2019-1,kovacs_2014}, we have for $z\in \Gamma^{\tau}_{\theta, \kappa}$,
\begin{equation}\label{eqH-2l2}
\|\tilde{\mathcal{E}}_{h}(z)\|_{\hat{H}^{-1}(D)\rightarrow L^{2}(D)}\leq\left\{
\begin{aligned}
&Ch|z|^{\alpha-1},\\
&C|z|^{\alpha/2-1}.
\end{aligned}\right .
\end{equation}
On the other hand, introducing $\mathbb{I}$ as  the identity operator, we have, for $v\in \hat{H}^{2}(D)$ and  $z\in \Gamma^{\tau}_{\theta, \kappa}$,
\begin{equation*}
	\begin{aligned}
		&\|\tilde{\bar{E}}(z)-\tilde{\bar{E}}_{h}(z)R_{h}v\|_{ \mathbb{H}}\\
		\leq &\left \|(\mathbb{I}-A((\delta_{\tau}(e^{-z\tau}))^{\alpha}+A)^{-1}\right. \\
		&~~\left. -R_{h}+A_{h}((\delta_{\tau}(e^{-z\tau}))^{\alpha}+A_{h})^{-1}R_{h})\frac{z\tau(\delta_{\tau}(e^{-z\tau}))^{-1}}{e^{z\tau}-1}v\right \|_{ \mathbb{H}}\\
		\leq &Ch^{2}|z|^{-1}\|Av\|_{\mathbb{H}},
	\end{aligned}
\end{equation*}
which leads to
\begin{equation*}
	\begin{aligned}
	\|\tilde{\mathcal{E}}_{h}(z)v\|_{\mathbb{H}}=&\|\tilde{\bar{E}}(z)v-\tilde{\bar{E}}_{h}(z)R_{h}v\|_{ \mathbb{H}}\\
	&+\|\tilde{\bar{E}}_{h}(z)R_{h}v-\tilde{\bar{E}}_{h}(z)P_{h}v\|_{ \mathbb{H}}\\
	=&\|\tilde{\bar{E}}(z)v-\tilde{\bar{E}}_{h}(z)R_{h}v\|_{ \mathbb{H}}+Ch^{2}\|\tilde{\bar{E}}_{h}(z)Av\|_{\mathbb{H}}\\
	\leq& Ch^{2}|z|^{-1}\|Av\|_{\mathbb{H}};
	\end{aligned}
\end{equation*}
that is
\begin{equation}\label{eqH2l2}
	\|\tilde{\mathcal{E}}_{h}(z)\|_{\hat{H}^{2}(D)\rightarrow L^{2}(D)}\leq Ch^{2}|z|^{-1}\quad {\rm for}~ z\in \Gamma^{\tau}_{\theta, \kappa}.
\end{equation}
\begin{theorem}\label{thmuspace}
	Let $u^{n}$ and $u^{n}_{h}$ be the solutions of Eqs. \eqref{eqfullscheme} and \eqref{eqsemischeme}, respectively. For $\rho\in[-1,\frac{1}{2}]$, there holds
	\begin{equation*}
			(\mathbb{E}\|u^{n}-u^{n}_{h}\|_{\mathbb{H}}^{2})^{1/2}\leq Ch^{\min(2-2\rho,\frac{2H}{\alpha}-2\rho,2)}.
	\end{equation*}
\end{theorem}
\begin{proof}
	According to \eqref{eqspterrep0}, the Cauchy-Schwarz inequality, and Lemma \ref{eqcorlem}, we have
	\begin{equation*}
		\begin{aligned}
			&	\mathbb{E}\|u^{n}-u^{n}_{h}\|_{\mathbb{H}}^{2}\\
			\leq& CH(2H-1)\int_{0}^{t_{n}}\int_{0}^{t_{n}}\left \langle\frac{1}{\tau}\sum_{i=1}^{n}\chi_{(t_{i-1},t_{i}]}(s)\int_{t_{i-1}}^{t_{i}}\mathcal{E}_{h}(t_{n}-\xi)d\xi Q^{1/2},\right .\\
				&\qquad\qquad\left .\frac{1}{\tau}\sum_{i=1}^{n}\chi_{(t_{i-1},t_{i}]}(r)\int_{t_{i-1}}^{t_{i}}\mathcal{E}_{h}(t_{n}-\xi)d\xi Q^{1/2}\right  \rangle |r-s|^{2H-2}drds\\
				\leq& C\left(\int_{0}^{t_{n}} (t_{n}-s)^{\varrho}\left \|\frac{1}{\tau}\sum_{i=1}^{n}\chi_{(t_{i-1},t_{i}]}(s)\int_{t_{i-1}}^{t_{i}}\mathcal{E}_{h}(t_{n}-\xi)d\xi Q^{1/2}\right \|_{\mathcal{L}_{2}}^{2}ds\right)^{1/2}\\
				\cdot&\left(\int_{0}^{t_{n}}(t_{n}-s)^{-\varrho} \right .\\
				&\qquad\qquad\left .\cdot\left \|\int_{s}^{t_{n}}\frac{1}{\tau}\sum_{i=1}^{n}\chi_{(t_{i-1},t_{i}]}(r)\int_{t_{i-1}}^{t_{i}}\mathcal{E}_{h}(t_{n}-\xi)d\xi Q^{1/2}|r-s|^{2H-2}dr\right\|_{\mathcal{L}_{2}}^{2} ds\right)^{1/2}\\
				\leq& C\uppercase\expandafter{\romannumeral1}\cdot \uppercase\expandafter{\romannumeral2}.
		\end{aligned}
	\end{equation*}
	 Here we consider the case $\rho\in[0,1/2]$ first. Choosing $\varrho=-1+2\alpha+\epsilon$ and using Lemma \ref{Lemseriesest}, \eqref{eql2l21}, \eqref{eqH-2l2}, and interpolation property give
	\begin{equation*}
\begin{aligned}
\uppercase\expandafter{\romannumeral1}^{2}\leq&C\int_{0}^{t_{n}}(t_{n}-s)^{-1+2\alpha+\epsilon}\left\|\frac{1}{\tau}\sum_{i=1}^{n}\chi_{(t_{i-1},t_{i}]}(s)\int_{t_{i-1}}^{t_{i}}\mathcal{E}_{h}(t_{n}-\xi)d\xi Q^{1/2}\right\|^{2}_{\mathcal{L}_{2}}ds\\
\leq& C\int_{0}^{t_{n}}(t_{n}-s)^{-1+2\alpha+\epsilon}\left (\frac{1}{\tau}\sum_{i=1}^{n}\chi_{(t_{i-1},t_{i}]}(s)\int_{t_{i-1}}^{t_{i}}\int_{\Gamma^{\tau}_{\theta,\kappa}}|e^{z(t_{n}-\xi)}||\delta_{\tau}(e^{-z\tau})|^{\alpha-1}\right .\\
&\left.\cdot\left\|((\delta_{\tau}(e^{-z\tau}))^{\alpha}+A)^{-1}-(((\delta_{\tau}(e^{-z\tau}))^{\alpha}+A_{h})^{-1}P_{h})A^{\rho}\right\|\left |\frac{z\tau}{e^{z\tau}-1}\right ||dz|d\xi\right )^{2}ds\\
\leq& Ch^{2(2-2\rho)}\int_{0}^{t_{n}}(t_{n}-s)^{-1+2\alpha+\epsilon}\\
&\qquad\qquad\cdot\left (\frac{1}{\tau}\sum_{i=1}^{n}\chi_{(t_{i-1},t_{i}]}(s)\int_{t_{i-1}}^{t_{i}}\int_{\Gamma^{\tau}_{\theta,\kappa}}|e^{z(s-\xi)}||e^{z(t_{n}-s)}||z|^{\alpha-1}|dz|d\xi\right )^{2}ds\\
\leq& Ch^{2(2-2\rho)}\int_{0}^{t_{n}} (t_{n}-s)^{-1+\epsilon}ds.\\
\end{aligned}
\end{equation*}
Similarly, by Lemma \ref{Lemseriesest}, \eqref{eql2l21}, \eqref{eql2l22}, \eqref{eqH-2l2}, and interpolation property, there is
\begin{equation*}
\begin{aligned}
&\left\|\int_{s}^{t_{n}}\frac{1}{\tau}\sum_{i=1}^{n}\chi_{(t_{i-1},t_{i}]}(r)\int_{t_{i-1}}^{t_{i}}\mathcal{E}_{h}(t_{n}-\xi)d\xi Q^{1/2}|r-s|^{2H-2}dr\right\|^{2}_{\mathcal{L}_{2}}\\
\leq& C\left (\int_{s}^{t_{n}}\frac{1}{\tau}\sum_{i=1}^{n}\chi_{(t_{i-1},t_{i}]}(r)\int_{t_{i-1}}^{t_{i}}\int_{\Gamma_{\theta,\kappa}}|e^{z(t_{n}-\xi)}||\delta_{\tau}(e^{-z\tau})|^{\alpha-1}\left |\frac{z\tau}{e^{z\tau}-1}\right |\right .\\
&\cdot\left\|((\delta_{\tau}(e^{-z\tau}))^{\alpha}+A)^{-1}-((\delta_{\tau}(e^{-z\tau}))^{\alpha}+A_{h})^{-1}P_{h})A^{\rho}\right \||dz|d\xi|r-s|^{2H-2}dr\Bigg )^{2}\\
\leq& Ch^{2(2-2\rho)\beta}\left (\int_{s}^{t_{n}}\frac{1}{\tau}\sum_{i=1}^{n}\chi_{(t_{i-1},t_{i}]}(r)\int_{t_{i-1}}^{t_{i}}\int_{\Gamma_{\theta,\kappa}}|e^{z(t_{n}-\xi)}|\right .\\
&\qquad\qquad\qquad\qquad\qquad\qquad\qquad\cdot|z|^{\alpha-(1-\rho)(1-\beta)\alpha-1}|dz|d\xi|r-s|^{2H-2}dr\Bigg )^{2}\\
\leq& Ch^{2(2-2\rho)\beta}\left (\int_{s}^{t_{n}}(t_{n}-r)^{(1-\rho)(1-\beta)\alpha-\alpha}(r-s)^{2H-2}dr\right )^{2}.\\
\end{aligned}
\end{equation*}
Therefore
\begin{equation*}
\begin{aligned}
\uppercase\expandafter{\romannumeral2}^{2}
\leq& Ch^{2(2-2\rho)\beta}\int_{0}^{t_{n}}(t_{n}-s)^{1-2\alpha-\epsilon+4H+2\alpha(1-\rho)(1-\beta)-2\alpha-2}ds,
\end{aligned}
\end{equation*}
	where we need to require that $\beta\in [0,1]$ and
		$	1-2\alpha-\epsilon+4H+2\alpha(1-\rho)(1-\beta)-2\alpha-2>-1$,
	i.e., $(2-2\rho)\beta<2-2\rho-\frac{4\alpha-4H+\epsilon}{\alpha}$.
	
	When $\rho\in[-1,0)$, taking $\varrho=-1+2(\rho+1)\alpha+\epsilon$ and using Lemma \ref{Lemseriesest}, \eqref{eql2l21}, \eqref{eqH2l2}, and the interpolation property imply
	\begin{equation*}
	\begin{aligned}
	\uppercase\expandafter{\romannumeral1}^{2}\leq&C\int_{0}^{t_{n}}(t_{n}-s)^{-1+2(\rho+1)\alpha+\epsilon}\left\|\frac{1}{\tau}\sum_{i=1}^{n}\chi_{(t_{i-1},t_{i}]}(s)\int_{t_{i-1}}^{t_{i}}\mathcal{E}_{h}(t_{n}-\xi)d\xi Q^{1/2}\right\|^{2}_{\mathcal{L}_{2}}ds\\
	\leq& C\int_{0}^{t_{n}}(t_{n}-s)^{-1+2(\rho+1)\alpha+\epsilon}\left (\frac{1}{\tau}\sum_{i=1}^{n}\chi_{(t_{i-1},t_{i}]}(s)\int_{t_{i-1}}^{t_{i}}\int_{\Gamma^{\tau}_{\theta,\kappa}}|e^{z(t_{n}-\xi)}||\delta_{\tau}(e^{-z\tau})|^{\alpha-1}\right .\\
	&\left\|((\delta_{\tau}(e^{-z\tau}))^{\alpha}+A)^{-1}-((\delta_{\tau}(e^{-z\tau}))^{\alpha}+A_{h})^{-1}P_{h})A^{\rho}\right\|\left |\frac{z\tau}{e^{z\tau}-1}\right ||dz|d\xi\Bigg )^{2}ds\\
	\leq& Ch^{4}\int_{0}^{t_{n}}(t_{n}-s)^{-1+2(\rho+1)\alpha+\epsilon}\\
	&\qquad\cdot\left (\frac{1}{\tau}\sum_{i=1}^{n}\chi_{(t_{i-1},t_{i}]}(s)\int_{t_{i-1}}^{t_{i}}\int_{\Gamma^{\tau}_{\theta,\kappa}}|e^{z(s-\xi)}||e^{z(t_{n}-s)}||z|^{(\rho+1)\alpha-1}|dz|d\xi\right )^{2}ds\\
	\leq& Ch^{4}\int_{0}^{t_{n}}(t_{n}-s)^{-1+2(\rho+1)\alpha+\epsilon}\left (\int_{\Gamma^{\tau}_{\theta,\kappa}}|e^{z(t_{n}-s)}||z|^{(\rho+1)\alpha-1}|dz|\right )^{2}ds\\
	\leq& Ch^{4}\int_{0}^{t_{n}} (t_{n}-s)^{-1+\epsilon}ds.\\
	\end{aligned}
	\end{equation*}
	Similarly, by Lemma \ref{Lemseriesest}, \eqref{eql2l21}, \eqref{eqH2l2}, and the interpolation property, we obtain
	\begin{equation*}
	\begin{aligned}
	&\left\|\int_{s}^{t_{n}}\frac{1}{\tau}\sum_{i=1}^{n}\chi_{(t_{i-1},t_{i}]}(r)\int_{t_{i-1}}^{t_{i}}\mathcal{E}_{h}(t_{n}-\xi)d\xi Q^{1/2}|r-s|^{2H-2}dr\right\|^{2}_{\mathcal{L}_{2}}\\
	\leq& C\left (\int_{s}^{t_{n}}\frac{1}{\tau}\sum_{i=1}^{n}\chi_{(t_{i-1},t_{i}]}(r)\int_{t_{i-1}}^{t_{i}}\int_{\Gamma_{\theta,\kappa}}|e^{z(t_{n}-\xi)}||\delta_{\tau}(e^{-z\tau})|^{\alpha-1}\left |\frac{z\tau}{e^{z\tau}-1}\right |\right .\\
	&\cdot\left\|((\delta_{\tau}(e^{-z\tau}))^{\alpha}+A)^{-1}-((\delta_{\tau}(e^{-z\tau}))^{\alpha}+A_{h})^{-1}P_{h})A^{\rho}\right \||dz|d\xi|r-s|^{2H-2}dr\Bigg )^{2}\\
	\leq& Ch^{4(1+\rho)\beta-4\rho}\left (\int_{s}^{t_{n}}\frac{1}{\tau}\sum_{i=1}^{n}\chi_{(t_{i-1},t_{i}]}(r)\int_{t_{i-1}}^{t_{i}}\int_{\Gamma_{\theta,\kappa}}|e^{z(t_{n}-\xi)}|\right .\\
	&\qquad\qquad\qquad\qquad\qquad\qquad\cdot|z|^{\alpha+\alpha\rho-(1+\rho)\alpha(1-\beta)-1}|dz|d\xi|r-s|^{2H-2}dr\Bigg )^{2}.\\
	\end{aligned}
	\end{equation*}
	So
	\begin{equation*}
	\begin{aligned}
	\uppercase\expandafter{\romannumeral2}^{2}\leq& Ch^{4(1+\rho)\beta-4\rho}\int_{0}^{t_{n}}(t_{n}-s)^{1-2(\rho+1)\alpha-\epsilon}\left (\int_{s}^{t_{n}}(t_{n}-r)^{-(1+\rho)\alpha\beta}(r-s)^{2H-2}dr\right )^{2}ds\\
	\leq& Ch^{4(1+\rho)\beta-4\rho}\int_{0}^{t_{n}}(t_{n}-s)^{1-2(\rho+1)\alpha-\epsilon+4H-2(1+\rho)\beta\alpha-2}ds,\\
	\end{aligned}
	\end{equation*}
	where we need to require that $1-2(\rho+1)\alpha-\epsilon+4H-2\beta\alpha(1+\rho)-2>-1$, i.e., $2\beta(1+\rho)<\frac{4H-\epsilon}{\alpha}-2(1+\rho)$
	and $\beta\in [0,1]$. Combining above estimates for $\uppercase\expandafter{\romannumeral1}$ and $\uppercase\expandafter{\romannumeral2}$, the desired result is reached.
\end{proof}

Similarly,  using the same way to discretize Eq. \eqref{equretosol1} leads to
\begin{equation}\label{eqsemischeme1}
\left \{
\begin{aligned}
&\left(\frac{v^{n}_h-v^{n-1}_{h}}{\tau},\nu_{h}\right)+\sum_{i=0}^{n-1}d^{(1-\alpha)}_{i}(A_hv^{n-i}_h,\nu_{h})=0\quad \forall \nu_{h}\in X_h,
\\
&v^{0}_h=P_{h}G_{0}. 
\\
\end{aligned}
\right.
\end{equation}
 Combining Theorem \ref{thmuspace} and the results provided in \cite{jin2016-1}, we have the estimate for $G^{n}$.
\begin{theorem}\label{ThmGspace}
	Let $G^{n}=u^{n}+v^{n}$ and $G^{n}_{h}=u^{n}_{h}+v^{n}_{h}$, where $u^{n}$ and $v^{n}$ are the solutions of Eqs. \eqref{eqfullscheme} and \eqref{eqfullscheme1} and $u^{n}_{h}$ and $v^{n}_{h}$ are the solutions of Eqs. \eqref{eqsemischeme} and \eqref{eqsemischeme1}, respectively. Let $\|A^{-\rho}\|_{\mathcal{L}^{0}_{2}}<\infty$ with $\rho\in[-1,1/2]$. Assume $G_{0}\in\hat{H}^{q} (D) $ with $q\leq 2$. Then we have
	\begin{equation*}
	(\mathbb{E}\|G^{n}-G^{n}_{h}\|_{\mathbb{H}}^{2})^{1/2}\leq Ch^{\min(2,2-2\rho,\frac{H}{\alpha}-2\rho)}+Ct^{q\alpha/2-1}h^{2}\|G_{0}\|_{\hat{H}^{q} (D)}.
	\end{equation*}
\end{theorem}
\section{Numerical experiments}
In this section, we perform numerical experiments to verify the effectiveness of the numerical schemes.
 Here we consider the covariance operator $Q$ that shares the eigenfunctions with the operator $A$ and denote the eigenvalues of $Q$ as $\varLambda_{k}=k^{m}$, $k=1,2,\cdots$. To get the relationship between $\rho$ and $m$, a theorem about the eigenvalues of $A$ is needed.
 \begin{theorem}[\cite{laptev_1997,Peter1983}]\label{thmeigenvalue}
 	Let $D$ denote a bounded domain in $\mathbb{R}^d \,(d=1,2,3)$ and $\varkappa_j$ be the $j$-th eigenvalue of the Dirichlet boundary problem for the Laplacian operator $-\Delta$ in $D$. Denote $|D|$ as the volume of $D$. We have that
 	\begin{equation*}
 	\varkappa_{j} \geq \frac{C_{d} d}{d+2} j^{2 / d}|D|^{-2 / d}
 	\end{equation*}
 	for all $j\geq 1$, where $C_{d}=(2 \pi)^{2} B_{d}^{-2 / d}$ and $B_d$ denotes the volume of the unit $d$-dimensional ball.
 \end{theorem}

 By the assumption $\|A^{-\rho}\|_{\mathcal{L}^{0}_{2}}<\infty$ and Theorem \ref{thmeigenvalue}, we get that $\rho>\frac{1+m}{4}d$. In our numerical experiments, we take the domain $D=(0,1)$. Thus the eigenfunctions of $Q$ are
 \begin{equation*}
 	\phi_{k}(x)=\sqrt{2}\sin(k\pi x).
 \end{equation*}
 We take
\begin{equation*}
W^{H}_{Q}(x,t)=\sum_{k=1}^{1000}\sqrt{\varLambda_{k}}\phi_k(x)W^{H}_k(t)
\end{equation*}
 and $100$ trajectories to calculate the solution of Eq. $\eqref{equretosol0}$ at time $T=0.01$. Since the exact solution $G$ is unknown, we use
\begin{equation*}
\begin{aligned}
&e_{h}=\left (\frac{1}{100}\sum_{i=1}^{100}\|G^{N}_{h}(\omega_{i})-G^{N}_{h/2}(\omega_{i})\|^{2}_{\mathbb{H}}\right )^{1/2},\\
&e_{\tau}=\left (\frac{1}{100}\sum_{i=1}^{100}\|G_{\tau}(\omega_{i})-G_{\tau/2}(\omega_{i})\|^{2}_{\mathbb{H}}\right )^{1/2}
\end{aligned}
\end{equation*}
to measure the spatial errors and temporal errors, where $G^{N}_{h}(\omega_{i})$ ($G_{\tau}(\omega_{i})$) means the numerical solution of $G$ at $t_N$ with mesh size $h$ (step size $\tau$) and trajectory $\omega_{j}$; and the spatial and temporal convergence rates can be, respectively, calculated by
\begin{equation*}
{\rm Rate}=\frac{\ln(e_{h}/e_{h/2})}{\ln(2)},\quad {\rm Rate}=\frac{\ln(e_{\tau}/e_{\tau/2})}{\ln(2)}.
\end{equation*}

\begin{example}
In this example, we take $G_{0}=0$ and $h=1/100$. The numerical results  for different $\alpha$, $H$, and $m$ are provided in Tables \ref{tab:tm0}, \ref{tab:tm-05}, and \ref{tab:tm-1}, where the numbers in the bracket in the last column denote the theoretical rates predicted by Theorem \ref{thmGtime}. When $m=0$, $\dot{W}^{H}_{Q}$ becomes a fractional Gaussian white noise and $\rho>\frac{1}{4}$ in one-dimensional domain, which leads to $\mathcal{O}(\tau^{H-\alpha/4})$ convergence in time and all results provided in Table \ref{tab:tm0} are consistent with the predicted results. When we improve regularity of the noise, i.e., $m=-0.5$ and $m=-1$, the convergence rates provided in Tables \ref{tab:tm-05} and \ref{tab:tm-1} are also improved and agree with the results provided in Theorem \ref{thmGtime}.
\begin{table}[htbp]
	\caption{Temporal errors and convergence rates with $m=0$.}
	\begin{tabular}{|c|c|cccc|c|}
		\hline
	$H$ & $\alpha\backslash T/\tau$ & 32 & 64 & 128 & 256 & Rate  \\
		\hline
	& 0.3 & 4.093E-04 & 2.698E-04 & 1.926E-04 & 1.303E-04 & $\approx$ 0.5503 (0.5250) \\
	0.6 & 0.5 & 9.714E-04 & 7.418E-04 & 5.304E-04 & 3.811E-04 & $\approx$ 0.4500 (0.4750) \\
	& 0.8 & 3.116E-03 & 2.363E-03 & 1.723E-03 & 1.313E-03 & $\approx$ 0.4154 (0.4000) \\
		\hline
	& 0.3 & 8.907E-05 & 5.627E-05 & 3.869E-05 & 2.212E-05 & $\approx$ 0.6698 (0.6750) \\
	0.75 & 0.5 & 2.260E-04 & 1.522E-04 & 9.789E-05 & 6.428E-05 & $\approx$ 0.6046 (0.6250) \\
	& 0.8 & 7.213E-04 & 4.945E-04 & 3.332E-04 & 2.218E-04 & $\approx$ 0.5670 (0.5500) \\
		\hline
	& 0.3 & 1.919E-05 & 1.214E-05 & 6.537E-06 & 3.653E-06 & $\approx$ 0.7978 (0.8250) \\
	0.9 & 0.5 & 4.857E-05 & 2.778E-05 & 1.651E-05 & 9.619E-06 & $\approx$ 0.7787 (0.7750) \\
	& 0.8 & 1.487E-04 & 8.969E-05 & 5.559E-05 & 3.345E-05 & $\approx$ 0.7174 (0.7000) \\
		\hline
	\end{tabular}
	\label{tab:tm0}
\end{table}

\begin{table}[htbp]
	\caption{Temporal errors and convergence rates with $m=-0.5$.}
	\begin{tabular}{|c|c|cccc|c|}
		\hline
		$H$ & $\alpha\backslash T/\tau$ & 32 & 64 & 128 & 256 & Rate  \\
		\hline
		& 0.3 & 3.781E-04 & 2.426E-04 & 1.679E-04 & 1.088E-04 & $\approx$ 0.5989 (0.5625) \\
		0.6 & 0.5 & 7.895E-04 & 5.241E-04 & 3.716E-04 & 2.638E-04 & $\approx$ 0.5271 (0.5375) \\
		& 0.8 & 1.714E-03 & 1.199E-03 & 8.417E-04 & 5.829E-04 & $\approx$ 0.5187 (0.5000) \\
		\hline
		& 0.3 & 8.254E-05 & 5.370E-05 & 3.166E-05 & 1.906E-05 & $\approx$ 0.7049 (0.7125) \\
		0.75 & 0.5 & 1.824E-04 & 1.158E-04 & 7.161E-05 & 4.355E-05 & $\approx$ 0.6887 (0.6875) \\
		& 0.8 & 4.124E-04 & 2.613E-04 & 1.663E-04 & 1.041E-04 & $\approx$ 0.6623 (0.6500) \\
		\hline
		& 0.3 & 1.716E-05 & 9.468E-06 & 5.379E-06 & 2.839E-06 & $\approx$ 0.8654 (0.8625) \\
		0.9 & 0.5 & 4.064E-05 & 2.127E-05 & 1.193E-05 & 6.746E-06 & $\approx$ 0.8636 (0.8375) \\
		& 0.8 & 8.630E-05 & 4.947E-05 & 2.861E-05 & 1.636E-05 & $\approx$ 0.7997 (0.8000) \\
		\hline
	\end{tabular}
	\label{tab:tm-05}
\end{table}

\begin{table}[htbp]
	\caption{Temporal errors and convergence rates with $m=-1$.}
	\begin{tabular}{|c|c|cccc|c|}
		\hline
		$H$ & $\alpha\backslash T/\tau$ & 32 & 64 & 128 & 256 & Rate  \\
		\hline
		& 0.3 & 3.637E-04 & 2.346E-04 & 1.565E-04 & 1.029E-04 & $\approx$ 0.6071 (0.6000) \\
		0.6 & 0.5 & 5.980E-04 & 3.949E-04 & 2.567E-04 & 1.684E-04 & $\approx$ 0.6093 (0.6000) \\
		& 0.8 & 9.719E-04 & 6.278E-04 & 4.184E-04 & 2.779E-04 & $\approx$ 0.6021 (0.6000) \\
		\hline
		& 0.3 & 8.365E-05 & 4.883E-05 & 3.111E-05 & 1.764E-05 & $\approx$ 0.7484 (0.7500) \\
		0.75 & 0.5 & 1.401E-04 & 8.730E-05 & 5.062E-05 & 3.131E-05 & $\approx$ 0.7207 (0.7500) \\
		& 0.8 & 2.442E-04 & 1.413E-04 & 8.307E-05 & 4.982E-05 & $\approx$ 0.7644 (0.7500) \\
		\hline
		& 0.3 & 1.916E-05 & 9.455E-06 & 5.350E-06 & 2.908E-06 & $\approx$ 0.9068 (0.9000) \\
		0.9 & 0.5 & 3.436E-05 & 1.787E-05 & 9.471E-06 & 5.105E-06 & $\approx$ 0.9168 (0.9000) \\
		& 0.8 & 5.986E-05 & 3.206E-05 & 1.712E-05 & 9.263E-06 & $\approx$ 0.8973 (0.9000) \\
		\hline
	\end{tabular}
	\label{tab:tm-1}
\end{table}

\end{example}

\begin{example}
	In this example, we take $G_{0}=x(1-x)$ and $\tau=T/1024$. The numerical results  for different $\alpha$, $H$, and $m$ are provided in Tables \ref{tab:sm0}, \ref{tab:sm-05}, \ref{tab:sm-1}, and \ref{tab:sm-15}.  In the tables, the numbers in the bracket in the last column denote the theoretical rates predicted by Theorem \ref{ThmGspace}. When $m=0$,  $m=-0.5$,  and $m=-1$, we have $\rho\geq 0$ and the convergence rates provided in Tables \ref{tab:sm0}, \ref{tab:sm-05}, and \ref{tab:sm-1} agree with Theorem \ref{ThmGspace}. Furthermore, to better show the effect of $\rho$ on the convergence rates, we take $m=-1.5$ and $\alpha\geq H=0.6$.  Thus $\rho>-1/4$ and the predicted convergence rates are $\mathcal{O}(h^{\frac{2H}{\alpha}-2\rho})$ according to Theorem \ref{ThmGspace}; the numerical results are presented in Table \ref{tab:sm-15}, which  agree with the predicted ones.
	\begin{table}[htbp]
		\caption{Spatial errors and convergence rates with $m=0$.}
		\begin{tabular}{|c|c|cccc|c|}
			\hline
			$H$ & $\alpha\backslash1/h$ & 16 & 32 & 64 & 128 & Rate \\
				\hline
			& 0.3 & 2.683E-04 & 9.078E-05 & 3.334E-05 & 1.183E-05 & $\approx$ 1.5012 (1.5000) \\
			0.6 & 0.5 & 7.776E-04 & 2.729E-04 & 1.029E-04 & 3.650E-05 & $\approx$ 1.4710 (1.5000) \\
			& 0.8 & 4.598E-03 & 2.211E-03 & 9.564E-04 & 3.839E-04 & $\approx$ 1.1941 (1.0000) \\
				\hline
			& 0.3 & 1.426E-04 & 4.907E-05 & 1.638E-05 & 5.801E-06 & $\approx$ 1.5399 (1.5000) \\
			0.75 & 0.5 & 3.358E-04 & 1.219E-04 & 4.359E-05 & 1.554E-05 & $\approx$ 1.4778 (1.5000) \\
			& 0.8 & 1.528E-03 & 6.483E-04 & 2.536E-04 & 9.527E-05 & $\approx$ 1.3346 (1.3750) \\
				\hline
			& 0.3 & 9.642E-05 & 2.877E-05 & 9.194E-06 & 3.038E-06 & $\approx$ 1.6627 (1.5000) \\
			0.9 & 0.5 & 1.956E-04 & 6.457E-05 & 2.173E-05 & 7.588E-06 & $\approx$ 1.5627 (1.5000) \\
			& 0.8 & 6.464E-04 & 2.509E-04 & 8.918E-05 & 3.196E-05 & $\approx$ 1.4460 (1.5000) \\
				\hline
		\end{tabular}
		\label{tab:sm0}
	\end{table}
	\begin{table}[htbp]
		\caption{Spatial errors and convergence rates with $m=-0.5$.}
		\begin{tabular}{|c|c|cccc|c|}
			\hline
			$H$ & $\alpha\backslash1/h$ & 16 & 32 & 64 & 128 & Rate \\
			\hline
			& 0.3 & 1.659E-04 & 4.954E-05 & 1.495E-05 & 4.427E-06 & $\approx$ 1.7426 (1.7500) \\
			0.6 & 0.5 & 4.426E-04 & 1.389E-04 & 4.235E-05 & 1.314E-05 & $\approx$ 1.6915 (1.7500) \\
			& 0.8 & 2.306E-03 & 9.030E-04 & 3.459E-04 & 1.244E-04 & $\approx$ 1.4042 (1.2500) \\
			\hline
			& 0.3 & 1.028E-04 & 2.927E-05 & 8.335E-06 & 2.382E-06 & $\approx$ 1.8104 (1.7500) \\
			0.75 & 0.5 & 2.244E-04 & 6.636E-05 & 1.981E-05 & 5.879E-06 & $\approx$ 1.7514 (1.7500) \\
			& 0.8 & 7.972E-04 & 2.860E-04 & 9.834E-05 & 3.297E-05 & $\approx$ 1.5319 (1.6250) \\
			\hline
			& 0.3 & 8.392E-05 & 2.169E-05 & 5.793E-06 & 1.563E-06 & $\approx$ 1.9156 (1.7500) \\
			0.9 & 0.5 & 1.492E-04 & 4.063E-05 & 1.138E-05 & 3.229E-06 & $\approx$ 1.8435 (1.7500) \\
			& 0.8 & 3.862E-04 & 1.216E-04 & 3.789E-05 & 1.176E-05 & $\approx$ 1.6791 (1.7500) \\
			\hline
		\end{tabular}
		\label{tab:sm-05}
	\end{table}

	\begin{table}[htbp]
		\caption{Spatial errors and convergence rates with $m=-1$.}
		\begin{tabular}{|c|c|cccc|c|}
			\hline
			$H$ & $\alpha\backslash1/h$ & 16 & 32 & 64 & 128 & Rate \\
			\hline
			& 0.3 & 1.220E-04 & 3.244E-05 & 8.596E-06 & 2.271E-06 & $\approx$ 1.9157 (2.0000) \\
			0.6 & 0.5 & 2.645E-04 & 7.669E-05 & 2.109E-05 & 5.764E-06 & $\approx$ 1.8401 (2.0000) \\
			& 0.8 & 1.262E-03 & 4.290E-04 & 1.413E-04 & 4.361E-05 & $\approx$ 1.6182 (1.5000) \\
			\hline
			& 0.3 & 8.428E-05 & 2.191E-05 & 5.663E-06 & 1.461E-06 & $\approx$ 1.9501 (2.0000) \\
			0.75 & 0.5 & 1.601E-04 & 4.293E-05 & 1.140E-05 & 3.013E-06 & $\approx$ 1.9106 (2.0000) \\
			& 0.8 & 4.800E-04 & 1.513E-04 & 4.464E-05 & 1.301E-05 & $\approx$ 1.7351 (1.8750) \\
			\hline
			& 0.3 & 7.896E-05 & 1.984E-05 & 5.008E-06 & 1.261E-06 & $\approx$ 1.9896 (2.0000) \\
			0.9 & 0.5 & 1.297E-04 & 3.310E-05 & 8.474E-06 & 2.172E-06 & $\approx$ 1.9669 (2.0000) \\
			& 0.8 & 2.442E-04 & 6.775E-05 & 1.880E-05 & 5.123E-06 & $\approx$ 1.8584 (2.0000) \\
			\hline
		\end{tabular}
		\label{tab:sm-1}
	\end{table}

\begin{table}[htbp]
	\caption{Spatial errors and convergence rates with $m=-1.5$.}
	\begin{tabular}{|c|c|cccc|c|}
		\hline
		$H$ & $\alpha\backslash1/h$ & 16 & 32 & 64 & 128 & Rate \\
		\hline
		& 0.7 & 4.402E-04 & 1.279E-04 & 3.444E-05 & 9.222E-06 & $\approx$ 1.8589(1.9643) \\
		0.6 & 0.8 & 6.860E-04 & 2.119E-04 & 6.128E-05 & 1.720E-05 & $\approx$ 1.7727(1.7500) \\
		& 0.9 & 1.062E-03 & 3.714E-04 & 1.177E-04 & 3.513E-05 & $\approx$ 1.6395(1.5833) \\
		\hline
	\end{tabular}
	\label{tab:sm-15}
\end{table}

\end{example}

\section{Conclusions}

We discuss the numerical methods for the stochastic fractional partial differential equation driven by external fractional Gaussian noise with the property of long-range dependence and positive correlation. The fully discrete scheme with the finite element approximation in space and the backward Euler convolution quadrature in time is developed. While analyzing the relationship between the regularity of the noise and convergence rates, we also provide the rigorous error analyses, verified by extensive numerical experiments.


\bibliographystyle{siamplain}
\bibliography{references}

\end{document}